\documentclass[10pt]{amsart}

\usepackage{mathtools} % mathtools loads amsmath. Load before amsthm, to ensure proper placement of \qedhere in align environments.
    \mathtoolsset{%
        mathic=true   % Italics correction in math mode. Sometimes causes an error saying "Unbalanced parantheses in the log file could not be resolved."
    }%
    \numberwithin{equation}{section}   % Number equation 2 in section 3 as (3.2)
\usepackage{amsthm}   % Load after amsmath, to ensure proper placement of \qedhere in align environments.
\usepackage{xcolor}   % Load before hyperref to have colors available for hyperref.
\usepackage{hyperref}
    \hypersetup{%
        bookmarksopen = true,    % Expand bookmarks in PDF viewer
        colorlinks    = true,    % Use colored text for links instead of surrounding them in boxes
        urlcolor      = cyan,    % Color for external hyperlinks.
        linkcolor     = magenta, % Color for internal links
        citecolor     = magenta, % Color for citations to the bibliography
    }%
\usepackage{enumitem}   % provides fine control of labels in lists the ability to [resume] lists. enumitem keeps bullet points from displaying in beamer.
\usepackage{subcaption}   % allows subfigures
\usepackage{amssymb}   % Provides \mathbb{} and \mathfrak{} fonts, among others. Loads amsfonts.
\usepackage{mathrsfs}   % provides \mathscr{} font
\usepackage{centernot} % provides the command \centernot, which prints a horizontally centered \not symbol.
\usepackage{graphicx}
    \graphicspath{{./figures/}}   % Set the path to graphics files. The doubled braces are necessary.
\usepackage{tikz}
    \usetikzlibrary{arrows.meta}
    \usetikzlibrary{shapes.misc,calc,math,shapes}
\usepackage{mathscinet} % Provides definitions for certain commands that occasionally occur in bibliographic data exported from MathSciNet and that are not provided by LaTeX.
\usepackage[noadjust]{cite}   % Has a bug where Theorem 1.1 ([9]) becomes Theorem 1.1 ( [9]). Fix by writing \hspace{1sp}\cite{...} or by using the [noadjust] option. See https://tex.stackexchange.com/a/164853
\usepackage[capitalize,noabbrev,nameinlink]{cleveref}
    \crefname{proposition}{Proposition}{Propositions}   % Required to generate plural forms. See Section 8.1.3 ("Automatic \newtheorem Definitions") of cleveref documentation.
    \crefname{const}{Construction}{Constructions}
    \crefname{subsection}{Subsection}{Subsections}   % Otherwise, cleveref calls subsections "sections"
    \crefformat{ineq}{#2Inequality~(#1)#3}   % To use, label an inequality as in \label[ineq]{ineq:1}, using the optional argument "ineq". Need to use \crefformat instead of \crefname to get the parentheses.
        \crefrangeformat{ineq}{Inequalities~(#3#1#4) to~(#5#2#6)}
        \crefmultiformat{ineq}{Inequalities~(#2#1#3)}%
            { and~(#2#1#3)}{, (#2#1#3)}{ and~(#2#1#3)}
    \crefname{subthm}{Theorem}{Theorems}   % To use, label a subtheorem as in \label[subthm]{subthm:1}, using the optional argument "subthm".
    \crefname{sublem}{Lemma}{Lemmas}   % To use, label a sublemma as in \label[sublem]{sublem:1}, using the optional argument "sublem".
    \crefname{subprop}{Proposition}{Propositions}   % To use, label a subprop as in \label[subprop]{subprop:1}, using the optional argument "subprop".
    \crefname{subcor}{Corollary}{Corollaries}   % To use, label a subcor as in \label[subcor]{subcor:1}, using the optional argument "subcor".
    \crefformat{eqns}{#2Equations~(#1)#3}   % To use, label a group of equations as in \label[eqns]{eqns:1}, using the optional argument "eqns". Need to use \crefformat instead of \crefname to get the parentheses.

\DeclarePairedDelimiter{\nparens}{\lparen}{\rparen} % non-automatically resizing parentheses
\newcommand{\parens}{\nparens*} % automatically resizing parentheses
\newcommand{\bigparens}{\nparens[\big]}

\newcommand{\Biggparens}{\nparens[\Bigg]}

\DeclarePairedDelimiter{\angles}{\langle}{\rangle}

\DeclarePairedDelimiter{\verts}{\lvert}{\rvert}
\DeclarePairedDelimiter{\Verts}{\lVert}{\rVert}

\DeclarePairedDelimiter{\lbrackrparen}{\lbrack}{\rparen}
\newcommand*{\bverts}[1]{\left\lvert #1 \right\rvert} % automatically resizing buffered |...|
\newcommand*{\blparenrbrack}[1]{\left\lparen #1 \right\rbrack} % automatically resizing buffered (...]
\newcommand{\ncard}{\verts} % non-automatically resizing
\newcommand{\card}{\ncard*} % automatically resizing

\newcommand{\inner}{\angles*}
\newcommand{\norm}{\Verts*}

\newcommand{\ropen}{\lbrackrparen*} % automatically resizing right-open segment with buffer space
\newcommand{\blopen}{\blparenrbrack} % automatically resizing left-open segment with buffer space
\newcommand{\st}{} % confirm that "\st" is available to use
\newcommand{\nst}[1][]{% non-automatically resizing "such that" symbol
    \mathchoice{\;}{\;}{\,}{\,} % the space before the "such that" symbol to use in, respectively, display, text, script, and scriptscript styles.
    \mathord{:}    % Use : for the "such that" symbol
    \allowbreak % Eat the following space if a line break follows the above "such that" symbol.
    \mathchoice{\;}{\;}{\,}{\,} % the space after the "such that" symbol to use in, respectively, display, text, script, and scriptscript styles.
    \mathopen{} % Don't interpret the "such that" symbol as something to be subtracted from if it's followed by a minus sign.
}%
\DeclarePairedDelimiterX{\nsetof}[1]{\lbrace}{\rbrace}{% non-automatically resizing
    \renewcommand{\st}{\nst[\delimsize]} % resized "such that" symbol
    #1
}

\newcommand{\setof}{\nsetof*} % automatically resizing
\newcommand{\bigsetof}{\nsetof[\big]}

\newcommand{\defing}[1]{\textup{\textbf{#1}}} % for defining terms
\newcommand{\Z}{\mathbb{Z}} % the integers
\newcommand{\Znn}{\Z_{\ge 0}} % nonnegative integers
\newcommand{\Zp}{\Z_{\ge 1}} % positive integers
\newcommand{\Q}{\mathbb{Q}} % the rationals
\newcommand{\R}{\mathbb{R}} % the reals
\newcommand{\Rnn}{\R_{\ge 0}} % nonnegative reals
\newcommand{\GL}{\mathrm{GL}} % general linear group
\newcommand*{\Zmod}[1]{\Z/#1\Z} % Z/nZ for parameter n
\newcommand{\Zmodn}{\Zmod{n}} % Z/nZ

\newcommand*{\bnd}[1]{\partial #1}   % relative boundary of a set
\newcommand*{\intr}[1]{#1^{\circ}}   % relative interior of a set
\newcommand*{\minusset}[1]{\setminus \setof{#1}} % to "set minus" a set written in set-builder or roster notation, e.g., a singleton set.
\DeclareMathOperator{\conv}{conv} % convex hull
\newcommand{\deftobe}{\coloneq} % := for defining new notation
\newcommand{\betodef}{\eqcolon} % =: 
\newcommand{\divides}{\mathrel{\smash{\big\vert}}} % as in "a divides b"
\newcommand{\notdivides}{\centernot\divides}

\newcommand{\maps}{\colon} % as in $f \maps A \to B$

\newcommand{\limp}{\Longrightarrow} % for "A implies B"
\newcommand{\lif}{\Longleftarrow} % for "A if B"
\newcommand{\liff}{\Longleftrightarrow} % for "A iff B"

\newcommand{\resto}[2]{{% we make the whole thing an ordinary symbol
   \left.\kern-\nulldelimiterspace % automatically resize the bar with \right
   #1 % the function
   \right\vert_{#2} % this is the delimiter
}}

\DeclareFontFamily{U}{mathb}{\hyphenchar\font45}
\DeclareFontShape{U}{mathb}{m}{n}{
      <5> <6> <7> <8> <9> <10> gen * mathb
      <10.95> mathb10 <12> <14.4> <17.28> <20.74> <24.88> mathb12
      }{}
\DeclareSymbolFont{mathb}{U}{mathb}{m}{n}
\DeclareMathSymbol{\lefttorightarrow}      {3}{mathb}{"FC}
\renewcommand{\emptyset}{\varnothing}
\renewcommand{\epsilon}{\varepsilon}
\renewcommand{\subset}{\subseteq}

\renewcommand{\bar}{\overline}
\renewcommand{\setminus}{\smallsetminus}

\renewcommand{\phi}{\varphi}   % Redefine \phi to be the variant "open circle" phi.

\newcommand*{\ehr}[1]{\operatorname{ehr}_{#1}}   % ehrhart function of a set
\newcommand*{\dualp}[1]{#1^{\vee}}   % dual of a polyhedron
\newcommand*{\bndp}[1]{\mathfrak{b}_{#1}}   % number of boundary lattice points
\newcommand*{\intp}[1]{\mathfrak{i}_{#1}}   % number of interior lattice points
\DeclareMathOperator{\aff}{aff}   % affine span
\DeclareMathOperator{\den}{den}   % denominator
\newcommand*{\ldist}[1]{\beta_{#1}}   % lattice distance
\newcommand*{\llength}{\bverts}   % lattice length
\newcommand{\F}{\mathcal{F}}   % set of facets
\newcommand{\Fi}{\intr{\F}}   % set of open facets
\newcommand{\E}{\mathcal{E}}   % set of faces
\newcommand{\Ei}{\intr{\E}}   % set of open faces
\newcommand{\B}{\mathcal{B}}   % boundary partition
\DeclareMathOperator{\relvol}{relvol}   % relative volume
\DeclareMathOperator{\vertx}{vert}   % set of vertices
\newcommand{\RR}{\mathcal{R}}   % set of reticular facets
\newcommand{\RRbar}{\bar{\RR}}   % set of non-reticular facets
\newcommand{\area}[1]{\operatorname{area}(#1)}   % area
\newcommand*{\prp}[1]{#1^{\perp}}   % ccw perpendicular vector
\newcommand*{\dett}[1]{d_{#1}}   % \dett{i,j} is determinant of matrix [a_{i} a_{j}]
\newcommand*{\PIP}[3]{\smash{P^{#1}_{#2, #3}}} % A PIP with denominator #1, with #2 interior points, and with #3 boundary points
\newcommand{\SSS}{\mathscr{S}}   % a family of solutions
\newcommand{\T}{\mathscr{T}}   % the set of rational triangles

\newcommand{\rinlinematrix}[4]{%
    \begin{bsmallmatrix*}[r]
         #1 & #2 \\
         #3 & #4
    \end{bsmallmatrix*}
}%
\newcommand{\cinlinematrix}[4]{%
    \begin{bsmallmatrix*}[c]
         #1 & #2 \\
         #3 & #4
    \end{bsmallmatrix*}
}%

\usepackage{array}   % for \newcolumntype macro
\newcolumntype{L}{>{$}l<{$}} % math-mode version of "l" column type
\newcolumntype{C}{>{$}c<{$}} % math-mode version of "c" column type
\newcolumntype{R}{>{$}r<{$}} % math-mode version of "r" column type

\theoremstyle{definition}   % For upright-font theorems
\newtheorem{theorem}{Theorem}[section]   % To count from beginning of section. E.g., Theorem 3.2 for the second theorem-counter-counted thing in Section 3.

\newtheorem*{theorem*}{Theorem}
\newtheorem{lemma}[theorem]{Lemma}

\newtheorem*{lemma*}{Lemma}
\newtheorem{corollary}[theorem]{Corollary}
\newtheorem{proposition}[theorem]{Proposition}
\newtheorem*{proposition*}{Proposition}

\theoremstyle{definition}
\newtheorem{example}[theorem]{Example}

\newtheorem{definition}[theorem]{Definition}

\newtheorem*{definition*}{Definition}

\newtheorem*{correct}{CORRECTION}

\theoremstyle{remark}
\newtheorem{remark}[theorem]{Remark}

\newtheorem*{claim*}{Claim}

\newenvironment{enumeratetfae}{%
    \begin{enumerate}[label=(\roman*)]
}{%
    \end{enumerate}
}%

\pgfdeclarelayer{background}
\pgfdeclarelayer{foreground}
\pgfdeclarelayer{grid}
\pgfdeclarelayer{lattice}
\pgfdeclarelayer{axes}
\pgfdeclarelayer{polygon1}
\pgfsetlayers{background, main, polygon1, grid, axes, lattice, foreground}

\tikzset{%
    every picture/.style={scale=1},
    blackpoint/.style={circle, inner xsep=1.5pt, inner ysep=0pt, fill=black},
    graypoint/.style={circle, inner xsep=1.5pt, inner ysep=0pt, fill=gray},
    axisstyle/.style={-latex, line width = 1pt},
    arrowstyle/.style={-latex, line width = 1pt},
    polygonstyle/.style={fill=black!25, thick, rounded corners=0.001pt},
    segmentstyle/.style={thick, rounded corners=0.001pt},
}

\newcommand*{\drawlattice}[4]{%
    \begin{pgfonlayer}{lattice}
        \begin{scope}
            \foreach \x in {#1,...,#2}{%
            \foreach \y in {#3,...,#4}{%
                \draw (\x,\y) node[blackpoint]{};
            }}%
        \end{scope}
    \end{pgfonlayer}
}%

\newcommand*{\drawlatticegrid}[5]{%
\begin{pgfonlayer}{grid}
    \draw[step=1/#5, gray, very thin] (#1, #3) grid(#2, #4);
\end{pgfonlayer}
    \begin{pgfonlayer}{lattice}
        \begin{scope}
            \foreach \x in {#1,...,#2}{%
            \foreach \y in {#3,...,#4}{%
                \draw (\x,\y) node[blackpoint]{};
            }}%
        \end{scope}
    \end{pgfonlayer}
}%

\newcommand*{\drawtriangle}[3]{%
    \filldraw[polygonstyle] #1 -- #2 -- #3 -- cycle;
}%
\newcommand*{\drawfourgon}[4]{%
    \filldraw[polygonstyle] #1 -- #2 -- #3 -- #4 -- cycle;
}%
\newcommand*{\drawfivegon}[5]{%
    \filldraw[polygonstyle] #1 -- #2 -- #3 -- #4 -- #5 -- cycle;
}%
\newcommand*{\drawsixgon}[6]{%
    \filldraw[polygonstyle] #1 -- #2 -- #3 -- #4 -- #5 -- #6 -- cycle;
}%

\title{%
    Boundaries of pseudointegral polygons%
}%

\author[T.~B.~McAllister]
{%
    Tyrrell B. McAllister%
}%
\address
{%
    Mathematics and Statistics \\
    University of Wyoming \\
    Laramie, WY 82071 \\
    USA%
}%
\email
{%
    tmcallis@uwyo.edu%
}%

\author[J.~S~Williford]
{%
    Jason S. Williford%
}%
\address
{%
    Mathematics and Statistics \\
    University of Wyoming \\
    Laramie, WY 82071 \\
    USA%
}%
\email
{%
    jwillif1@uwyo.edu%
}%


\date{}

\begin{document}

\begin{abstract}
    We prove that a rational pseudointegral triangle with exactly one
    lattice point in its interior has at most $9$ lattice points on
    its boundary, where a polygon $P$ is called pseudointegral if the
    Ehrhart function of $P$ is a polynomial.  We further show that
    such a triangle never has exactly~$7$ lattice points on its
    boundary.  Our results determine the set of all Ehrhart
    polynomials of rational triangles with one interior lattice point.
    
    In addition, we construct convex pseudointegral polygons with~$i$
    interior lattice points and~$b$ boundary lattice points for all
    positive integral values of $(i,b)$ such that $b \le 5i + 4$.
    This is in contrast to integral polygons, which must satisfy $b
    \le 2i + 7$ by a result of Scott.  Our constructions yield many
    new Ehrhart polynomials of rational polygons in the $i \ge 2$
    case.
\end{abstract}

% % Version of abstract to paste into textfields when submitting:
% % 
% We prove that a rational pseudointegral triangle with exactly one lattice point in its interior has at most $9$ lattice points on its boundary, where a polygon $P$ is called pseudointegral if the Ehrhart function of $P$ is a polynomial. We further show that such a triangle never has exactly $7$ lattice points on its boundary. Our results determine the set of all Ehrhart polynomials of rational triangles with one interior lattice point.
% 
% In addition, we construct convex pseudointegral polygons with $i$ interior lattice points and $b$ boundary lattice points for all positive integral values of $(i,b)$ such that $b \le 5i + 4$. This is in contrast to integral polygons, which must satisfy $b \le 2i + 7$ by a result of Scott. Our constructions yield many new Ehrhart polynomials of rational polygons in the $i \ge 2$ case.

\maketitle

\section{Introduction}

A convex polytope $P \subset \R^{d}$ is \defing{reflexive} if both $P$
and its dual are integral polytopes.  In dimension~$2$, the reflexive
polygons are easy to recognize on sight, because they are precisely
the convex integral polygons in which the origin is the only interior
lattice point.  With this criterion, it is straightforward to show by
case analysis that there are precisely $16$ reflexive polygons, up to
automorphisms of the integer lattice.  See \cref{fig:reflexive
polygons} for representatives of these $16$~reflexive polygons.  One
observes that every reflexive polygon has at most~$9$ lattice points
on its boundary.  It follows that every convex integral polygon with
exactly one lattice point in its interior has at most~$9$ lattice
points on its boundary.

\begin{figure}
%     Haase & Schicho arrangement:
    \newcommand{\cax}{1}
    \newcommand{\cay}{1}
    \newcommand{\cbx}{5}
    \newcommand{\cby}{1}
    \newcommand{\ccx}{9}
    \newcommand{\ccy}{1}
    \newcommand{\cdx}{13}
    \newcommand{\cdy}{1}
    \newcommand{\cex}{16}
    \newcommand{\cey}{1}
    \newcommand{\cfx}{1}
    \newcommand{\cfy}{4}
    \newcommand{\cgx}{4}
    \newcommand{\cgy}{4}
    \newcommand{\chx}{7}
    \newcommand{\chy}{4}
    \newcommand{\cix}{10}
    \newcommand{\ciy}{4}
    \newcommand{\cjx}{13}
    \newcommand{\cjy}{4}
    \newcommand{\ckx}{16}
    \newcommand{\cky}{4}
    \newcommand{\clx}{1}
    \newcommand{\cly}{7}
    \newcommand{\cmx}{4}
    \newcommand{\cmy}{7}
    \newcommand{\cnx}{7}
    \newcommand{\cny}{7}
    \newcommand{\cox}{10}
    \newcommand{\coy}{7}
    \newcommand{\cpx}{15}
    \newcommand{\cpy}{7}
    \resizebox{\textwidth}{!}{
        \begin{tikzpicture}
            \drawlattice{0}{17}{0}{8}
            \drawtriangle{(\cax-1,\cay-1)}{(\cax+2,\cay-1)}{(\cax-1,\cay+2)}
            \drawfourgon{(\cbx-1,\cby-1)}{(\cbx+2,\cby-1)}{(\cbx,\cby+1)}{(\cbx-1,\cby)}
            \drawtriangle{(\ccx-1,\ccy-1)}{(\ccx+2,\ccy-1)}{(\ccx+1,\ccy+1)}
            \drawfourgon{(\cdx-1,\cdy-1)}{(\cdx+1,\cdy-1)}{(\cdx+1,\cdy)}{(\cdx,\cdy+1)}
            \drawfourgon{(\cex-1,\cey)}{(\cex,\cey-1)}{(\cex+1,\cey)}{(\cex,\cey+1)}
            \drawsixgon{(\cfx-1,\cfy)}{(\cfx,\cfy-1)}{(\cfx+1,\cfy-1)}{(\cfx+1,\cfy)}{(\cfx,\cfy+1)}{(\cfx-1,\cfy+1)}
            \drawfivegon{(\cgx-1,\cgy)}{(\cgx,\cgy-1)}{(\cgx+1,\cgy-1)}{(\cgx+1,\cgy)}{(\cgx,\cgy+1)}
            \drawfivegon{(\chx-1,\chy-1)}{(\chx,\chy-1)}{(\chx+1,\chy)}{(\chx+1,\chy+1)}{(\chx-1,\chy+1)}
            \drawfourgon{(\cix-1,\ciy-1)}{(\cix+1,\ciy-1)}{(\cix+1,\ciy+1)}{(\cix,\ciy+1)}
            \drawfivegon{(\cjx-1,\cjy-1)}{(\cjx+1,\cjy-1)}{(\cjx+1,\cjy)}{(\cjx,\cjy+1)}{(\cjx-1,\cjy)}
            \drawfourgon{(\ckx-1,\cky-1)}{(\ckx+1,\cky-1)}{(\ckx+1,\cky+1)}{(\ckx-1,\cky+1)}
            \drawtriangle{(\clx-1,\cly-1)}{(\clx+1,\cly)}{(\clx,\cly+1)}
            \drawtriangle{(\cmx-1,\cmy+1)}{(\cmx+1,\cmy+1)}{(\cmx,\cmy-1)}
            \drawfourgon{(\cnx-1,\cny)}{(\cnx,\cny-1)}{(\cnx+1,\cny-1)}{(\cnx,\cny+1)}
            \drawfourgon{(\cox-1,\coy-1)}{(\cox+2,\coy-1)}{(\cox,\coy+1)}{(\cox-1,\coy+1)}
            \drawtriangle{(\cpx-2,\cpy+1)}{(\cpx+2,\cpy+1)}{(\cpx,\cpy-1)}
        \end{tikzpicture}
    }
    \caption{%
        The $16$ reflexive polygons (adapted from \cite{HaaSch2009}).
    }%
    \label{fig:reflexive polygons}
\end{figure}
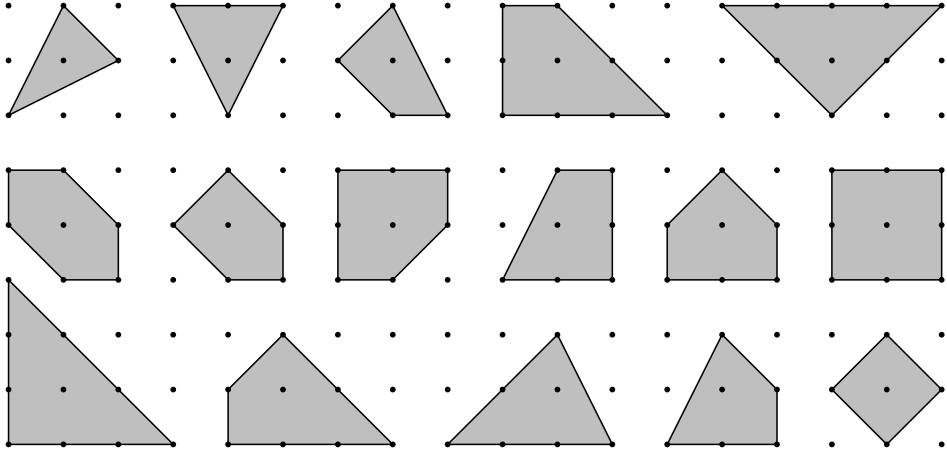

Our first main result (\cref{thm: first main result} below) provides
conditions under which the hypothesis of \emph{integrality} in this
upper bound may be weakened to \emph{pseudointegrality}.  A
\defing{pseudointegral polytope}, or \defing{PIP} for short, is a
polytope $P \subset \R^{d}$ whose \defing{Ehrhart function} $t \mapsto
\card{tP \cap \Z^{d}}$ is a polynomial function of $t \in \Zp$, where
$tP \deftobe \setof{ta \st a \in P}$.  In particular, all integral
polytopes are pseudointegral by a famous result of Ehrhart.  (We
review Ehrhart theory in \cref{sec: preliminaries}.)  We show that the
upper bound of at most~$9$ boundary lattice points holds for all
rational pseudo\-integral triangles that contain exactly one interior
lattice point.  More precisely, we prove the \mbox{following}.

\begin{theorem}\label{thm: first main result}
    Let $T \subset \R^{2}$ be a rational pseudointegral triangle that
    contains exactly one lattice point in its interior, and let $b$ be
    the number of lattice points on the boundary of $T$.  Then ${1 \le
    b \le 9}$, but $b \ne 7$.
\end{theorem}

It remains open whether every pseudointegral $n$-gon with $n \ge 4$
has at most~$9$ boundary lattice points.  We prove \cref{thm: first
main result} by reducing it to the following purely number-theoretic
statement, which may be of independent interest.

\begin{lemma}\label{lem: number-theory lemma for triangular PRPs}
    Let $x, y, z \in \Zp$ and $b \deftobe \frac{(x+y+z)^{2}}{xyz}$.
    If $b \in \Z$, then $b \le 9$ and $b \ne 7$.
\end{lemma}

The reduction of \cref{thm: first main result} to \cref{lem:
number-theory lemma for triangular PRPs} appears in \cref{sec:
reduction to number-theoretic lemma}.  The proof of \cref{lem:
number-theory lemma for triangular PRPs}, which completes the proof of
\cref{thm: first main result}, is in \cref{sec: Number-theoretic
lemmas for PRPs}.  A~generalization of \cref{lem: number-theory lemma
for triangular PRPs} is given in the appendix.

Conversely, in our second main result, we construct infinitely many
pseudo\-integral triangles with one interior lattice point and $b$
boundary lattice points for each value of $b$ allowed by \cref{thm:
first main result}.  This result answers a question of Bohnert
\cite[Remark~5.9]{Boh2024preprint}.  Recall that the
\defing{denominator} of a rational polytope $P$ is the minimum
positive integer $k$ such that $kP$ is integral.

\begin{theorem}\label{thm: Infinitely many triangular PRPs}
    Let $b \in \Z$ be such that $1 \le b \le 9$ and $b \ne 7$.  Then,
    for all $k \in \Z$, there exists a rational pseudointegral
    triangle with denominator $> k$ that has exactly one interior
    lattice point and $b$ boundary lattice points.
\end{theorem}

We construct these PIPs in \cref{sec: Infinite families of
pseudo-reflexive triangles} using suitable solutions to the equation $b
= (x+y+z)^{2}/(xyz)$ in \cref{lem: number-theory lemma for triangular
PRPs}.  For example, we find the following infinite family of PIPs,
each with one interior lattice point and $9$ boundary lattice points.

\begin{example}\label{exm: Fibonacci PRPs}
    Let $F_{1}, F_{2}, \dotsc$ be the Fibonacci sequence, with $F_{1}
    = F_{2} = 1$. Then
    \begin{equation*}
        T_{j}
            \deftobe%
            \conv{%
                \setof{%
                    \begin{bmatrix*}[c]
                        -3 F_{2j-1}/F_{2j+1}     \\
                         3 F_{2j-1}/F_{2j+1} - 1
                    \end{bmatrix*},
                    \begin{bmatrix*}[r]
                         0 \\
                        -1
                    \end{bmatrix*},
                    \begin{bmatrix*}[c]
                         3 F_{2j+1}/F_{2j-1} \\
                        -1
                    \end{bmatrix*}
                }%
            }%
    \end{equation*}
    is a PIP with exactly one interior lattice point and $9$ boundary
    lattice points for all $j \in \Zp$.
\end{example}

The authors did not anticipate that the Fibonacci sequence would arise
in such infinite sequences of rational PIPs.  However, with the
benefit of hindsight, this fact seems to be in accordance with the
appearance of the golden ratio in the irrational PIPs constructed by
Cristofaro-Gardiner, Li, and Stanley in
\cite[Example~1.2]{CriLiSta2019}.

It follows from \cref{thm: first main result} that pseudointegral
triangles with one interior lattice point cannot have a greater number
of boundary lattice points than is attained by some integral polygon
with one interior lattice point.  In contrast, our third main result
(\cref{thm: third main result} below) is a construction of
pseudointegral polygons with $i \ge 2$ interior lattice points that,
asymptotically in~$i$, have $5/2$~times as many boundary lattice
points as is attained by any integral polygon with $i$ interior
lattice points.

In order to state our third main result concisely, we introduce some
notation.  A~\defing{polygonal} PIP is a PIP that is a polygon.  Given
a convex polygon $P \subset \R^{2}$, let ${\intp{P} \deftobe
\ncard{\intr{P} \cap \Z^{2}}}$ and ${\bndp{P} \deftobe \ncard{\bnd{P}
\cap \Z^{2}}}$, where~$\intr{P}$, respectively~$\bnd{P}$, denotes the
topological interior, respectively boundary, of~$P$.

When $P$ is a polygonal PIP, the values of $\intp{P}$ and $\bndp{P}$
determine the coefficients of the Ehrhart polynomial $\ehr{P}(t)
\deftobe \card{tP \cap \Z^{2}}$ of~$P$ according to Pick's formula:
$\ehr{P}(t) = (\intp{P} + \tfrac{1}{2} \bndp{P} - 1) t^{2} +
\tfrac{1}{2} \bndp{P} t + 1$ \cite[Theorem~3.1]{McAWoo2005}.  For
integral polygons~$P$, the fact that $\bndp{P} \le 9$ when $\intp{P} =
1$ is a special case of a more general bound due to Paul~Scott, which
yields a characterization of all Ehrhart polynomials of integral
polygons.

\begin{theorem}[\cite{Sco1976}]\label{thm: Scott's inequality}
    For all $i, b \in \Zp$, %
    there exists an integral polygon~$P$ such that $\intp{P} = i$ and
    $\bndp{P} = b$ if and only if either $i = 1$ and $3 \le b \le 9$,
    or $i \ge 2$ and $3 \le b \le 2i + 6$.
\end{theorem}

Until recently, it was not known whether any polygonal PIP $P$ with
$\intp{P} \ge 2$ satisfied ${\bndp{P} > 2 \intp{P} + 6}$.  However,
exciting work by M.~Bohnert \cite{Boh2024preprint} has now ``broken
the Scott barrier'' by constructing half-integral polygonal PIPs $P$
with $\intp{P} \ge 2$ such that ${\bndp{P} = 2\intp{P} + 7}$.
(A~polytope $P \subset \R^{d}$ is \defing{half integral} if $2P$ is
integral.)  Moreover, Bohnert proved that his constructions attain the
maximum possible number of boundary points for half-integral polygonal
PIPs, in the sense that $\bndp{P} \le 2\intp{P} + 7$ for all
half-integral polygonal PIPs $P$.  Thus, Bohnert characterized
precisely the Ehrhart polynomials of half-integral polygonal PIPs,
just as Scott had done for integral polygons.

\begin{theorem}[{\cite[Theorem 6.4]{Boh2024preprint}}]
    For all $i, b \in \Zp$, there exists a half-integral polygonal PIP
    $P$ such that $\intp{P} = i$ and $\bndp{P} = b$ if and only if ${2
    \le b \le 2i + 7}$.
\end{theorem}

Our constructions below substantially extend the range of values of
$(\intp{P},\bndp{P})$ that are known to be realized by polygonal PIPs.
Thus we correspondingly extend the set of polynomials that are known
to be the Ehrhart polynomials of rational polygons.

\begin{theorem}\label{thm: third main result}
    For all $i, b \in \Zp$ such that $b \le 5i + 4$, there exists a
    rational polygonal PIP~$P$ such that $\intp{P} = i$ and $\bndp{P}
    = b$.
\end{theorem}

This theorem is proved in \cref{sec: New Ehrhart polynomials}.  The
vertices of the polygons that we require to prove \cref{thm: third
main result} have all of their vertices in $(\frac{1}{2}\Z^{2}) \cup
(\frac{1}{5}\Z^{2})$ and hence have denominator at most $10$.  These
results suggest the possibility that every \mbox{denominator-$d$}
polygonal PIP $P$ with $\intp{P} \ge 1$ must satisfy $\bndp{P} \le
\ell_{d}(\intp{P})$ for some linear polynomial $\ell_{d}(t)$ depending
only on $d$.  By the work of Scott and Bohnert, this claim is now
known to be true in the $d = 1, 2$ cases with $\ell_{1}(t) = 2t + 6$
and $\ell_{2}(t) = 2t + 7$.  The constructions in \cref{sec: New
Ehrhart polynomials} show that, if such polynomials $\ell_{d}(t)$
exist in the cases $d = 3, 4, 10$, then the leading coefficient of
$\ell_{3}(t)$ is $\ge 3$, of $\ell_{4}(t)$ is $\ge 4$, and of
$\ell_{10}(t)$ is $\ge 5$.

The remaining sections are organized as follows.  In \cref{sec:
preliminaries}, we review standard notation and results from Ehrhart
theory.  In \cref{sec: pseudoreflexive polygons}, we prove that all
polygonal PIPs with exactly one interior lattice point satisfy certain
standard properties of reflexive polygons.  In \cref{sec: reduction to
number-theoretic lemma}, we show that \cref{thm: first main result}
reduces to \cref{lem: number-theory lemma for triangular PRPs}, which
in turn we prove in \cref{sec: Number-theoretic lemmas for PRPs}.  In
\cref{sec: Infinite families of pseudo-reflexive triangles}, we
construct infinitely many pseudointegral PIPs with every attainable
number of boundary lattice points.  In \cref{sec: New Ehrhart
polynomials}, we give the constructions that prove \cref{thm: third
main result}.  Finally, in the appendix, we prove a generalization of
\cref{lem: number-theory lemma for triangular PRPs}.

\section{Preliminaries}
\label{sec: preliminaries}

In this section, we review standard terminology and results that we
will use in the following sections.  See especially \cite{BecRob2015}
for an excellent introduction to Ehrhart theory

Fix $d \in \Zp$ and let $\inner{\cdot, \cdot}$ be the standard inner
product on $\R^{d}$.  Given a subset $X \subset \R^{d}$, the
\defing{affine span} $\aff(X)$ of $X$ is the minimum affine subspace
of~$\R^{d}$ containing~$X$.  The set $X$ is \defing{reticular} if
${\aff(X) \cap \Z^{d} \ne \emptyset}$.  We write $\intr{X}$ and
$\bnd{X}$ for the \defing{relative interior} and \defing{relative
boundary}, respectively, of $X$, \emph{i.e.}, the interior and
boundary, respectively, of $X$ relative to $\aff(X)$.  Let $\intp{X}
\deftobe \card{\smash{\intr{X} \cap \Z^{d}}}$ and $\bndp{X} \deftobe
\card{\smash{\bnd{X} \cap \Z^{d}}}$.  If $X$ is measurable and
$\aff(X)$ is a translation of a rational linear subspace of $\R^{d}$,
then the \defing{relative volume} $\relvol(X)$ is the volume of $X -
a$ relative to the lattice $(\aff(X) - a) \cap \Z^{d}$, where $a$ is
any point in $\aff(X)$.  In particular, if $E \subset \R^{d}$ is a
line segment (open, half-open, or closed) that is parallel to a
one-dimensional rational linear subspace of $\R^{d}$, then we write
$\llength{E} \deftobe \relvol(E)$ and call $\llength{E}$ the
\defing{lattice length} of $E$.

A \defing{polytope} is the convex hull of a finite subset of $\R^{d}$.
Let $P \subset \R^{d}$ be a polytope.  The \defing{dimension} of $P$
is $\dim(P) \deftobe \dim(\aff(P))$.  A \defing{face} $F$ of $P$ is a
subset of the form $P \cap H$, where $H = \setof{a \in \R^{d} \st
\inner{u, a} = \beta}$ for some $u \in \R^{d}$ and $\beta \in \R$ such
that $\inner{u, a} \le \beta$ for all $a \in P$.  The \defing{facets}
of $P$ are the $(\dim(P) - 1)$-dimensional faces, and the
\defing{vertices} of $P$ are the $0$-dimensional faces.  Write
$\vertx(P)$ for the set of vertices of $P$.  The polytope $P$ is
\defing{rational}, respectively \defing{integral}, if $\vertx(P)
\subset \Q^{d}$, respectively $\vertx(P) \subset \Z^{d}$.  If $P$ is
rational, the \defing{denominator} of $P$ is the minimum positive
integer $k$ such that $kP$ is integral.  Write $\den(P)$ for the
denominator of $P$.

For every bounded subset $X$ of $\R^{d}$, the \defing{Ehrhart
function} of $X$ is the function $\ehr{X} \maps \Zp \to \Znn$ defined
by $\ehr{X}(t) \deftobe \ncard{tX \cap \Z^{d}}$, where $tX \deftobe
\setof{ta \st a \in X}$ for $t \in \Zp$.

\begin{theorem}[Ehrhart \cite{Ehr1962a}]\label{thm: Ehrhart's theorem}
    Let $P \subset \R^{d}$ be a nonempty rational polytope.  Then
    $\ehr{P}$ is a quasipolynomial of degree $\dim(P)$.  In
    particular, there exist unique periodic functions $c_{0}, c_{1},
    \dotsc, c_{\dim(P)} \maps \Z \to \Q$ such that
    \begin{equation}\label{eq: quasipolynomial of a polytope}
        \ehr{P}(t) 
            =%
            c_{0}(t) + c_{1}(t) t + \dotsb + c_{\dim(P)}(t)
            t^{\dim(P)}
    \end{equation}
    with $c_{\dim(P)}(t) = \relvol(P)$ if $tP$ is reticular, and
    $c_{\dim(P)}(t) = 0$ otherwise, for all $t \in \Zp$.
\end{theorem}

The right-hand side of \cref{eq: quasipolynomial of a polytope}
provides a canonical extension of the domain of $\ehr{P}$ to the
entirety of $\Z$.  Moreover the values of $\ehr{P}(t)$ for
negative~$t$ have a combinatorial interpretation due to
\emph{Ehrhart--Macdonald reciprocity}.

\begin{theorem}[Ehrhart--Macdonald reciprocity]
    Let $P \subset \R^{d}$ be a rational polytope.  Then the Ehrhart
    function of $\intr{P}$ is
    \begin{equation*}
        \ehr{\intr{P}}(t) 
            =
            (-1)^{\dim(P)}\ehr{P}(-t).
    \end{equation*}
\end{theorem}

Ehrhart proved that the Ehrhart functions of integral polytopes have a
particularly nice form:

\begin{theorem}\label{thm: Ehrhart's theorem for integral polytopes}
    If $P \subset \R^{d}$ is an integral polytope, then $\ehr{P}$ is a
    polynomial.
\end{theorem}

This property of integral polytopes motivates the following
definition.

\begin{definition}
    A \defing{pseudointegral polytope} (or \defing{PIP}) is a polytope
    $P \subset \R^{d}$ such that $\ehr{P}$ is a polynomial.
\end{definition}

A PIP $P$ is \defing{polygonal}, respectively \defing{triangular}, if
$P$ is a polygon, respectively triangle, in $\R^{2}$.  In the case of
a rational polygonal PIP $P$, the Ehrhart polynomial of~$P$ is
determined by the pair $(\intp{P}, \bndp{P})$ and vice versa.  This
follows from the following characterization of polygonal PIPs.

\begin{theorem}[\protect{\cite[Theorem~3.1]{McAWoo2005}}]
\label{thm: PIP criteria}%
    Let $P \subset \R^{2}$ be a rational polygon.  Then the following 
    are equivalent.
    \begin{enumeratetfae}
        \item
        $P$ is a PIP.
        
        \item
        $
            \ehr{P}(t) 
                =%
                \area{P} t^{2}
                +
                \tfrac{1}{2} \bndp{P} t 
                +
                1.
        $
        
        \item
        $\bndp{tP} = t \bndp{P}$ and $tP$ satisfies Pick's formula
        $\area{tP} = \intp{tP} + \tfrac{1}{2} \bndp{tP} - 1$ for all
        $t \in \Zp$.
    \end{enumeratetfae}
\end{theorem}

Every integral polytope is a PIP by \cref{thm: Ehrhart's theorem for
integral polytopes}.  However, infinitely many non-integral PIPs exist.
Indeed, even irrational PIPs exist \cite{CriLiSta2019}.  Moreover, the
set of Ehrhart polynomials includes polynomials that arise only as the
Ehrhart functions of non-integral PIPs; see \cref{exm: PIPs with 1 or 2
boundary points} below.

\begin{theorem}[\protect{\cite[Theorem~1.2]{McAMor2017}}]
    Given integers $i \ge 1$ and $b \in \setof{1, 2}$, there exists a
    polygonal PIP $P$ with $(\intp{P}, \bndp{P}) = (i, b)$.  However,
    there does not exist a polygonal PIP $P$ with $\bndp{P} = 0$.
\end{theorem}

The polygons constructed in \cite{McAMor2017} to prove this theorem
appear in slightly modified form in \cref{exm: PIPs with 1 or 2
boundary points} below.

\begin{example}\label{exm: PIPs with 1 or 2 boundary points}\mbox{}
    \begin{itemize}
        \item  
        For $i \in \Zp$, let
        $%
            P
                \deftobe%
                \conv
                \bigsetof{
                    (i, 0),
                    \bigparens{\frac{-2i}{2i+1}, \frac{2i-1}{2i+1}},
                    \bigparens{\frac{-1}{2i+1}, -\frac{2i-1}{2i+1}}
                }
        $. %
        Then $P$ is a PIP with ${\intp{P} = i}$ and ${\bndp{P} = 1}$.
        See \cref{subfig: PIP with b=1} for the $i = 2$ case.

        \item  
        For $i \in \Zp$, let 
        $
            P
                \deftobe%
                \conv
                \bigsetof{
                    (i, 0),
                    \bigparens{-1, \pm \frac{i}{i+1}}
                }
        $. %
        Then $P$ is a PIP with $\intp{P} = i$ and $\bndp{P} = 2$.  See
        \cref{subfig: PIP with b=2} for the $i=2$ case.
    \end{itemize}
    The triangles in these families all have Ehrhart polynomials that
    are not the Ehrhart polynomial of any integral polygon, since no
    integral polygon has fewer than $3$ lattice points on its
    boundary.
    
    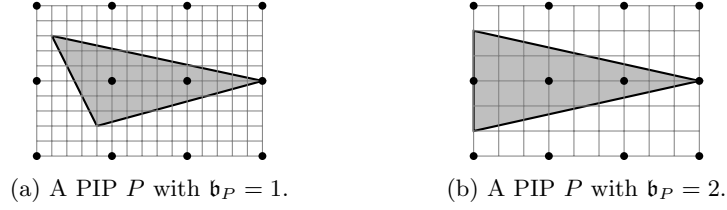
\begin{figure}
        \begin{subfigure}{0.45\textwidth}
            \centering
            \begin{tikzpicture} % i = 2
                \drawlatticegrid{-1}{2}{-1}{1}{5}
                \drawtriangle{(2,0)}{(-4/5,3/5)}{(-1/5,-3/5)}
            \end{tikzpicture}
            \caption{A PIP $P$ with $\bndp{P} = 1$.}
            \label{subfig: PIP with b=1}
        \end{subfigure}
        \begin{subfigure}{0.45\textwidth}
            \centering
            \begin{tikzpicture} % i = 2
                \drawlatticegrid{-1}{2}{-1}{1}{3}
                \drawtriangle{(2,0)}{(-1,2/3)}{(-1,-2/3)}
            \end{tikzpicture}
            \caption{A PIP $P$ with $\bndp{P} = 2$.}
            \label{subfig: PIP with b=2}
        \end{subfigure}%
        \caption{PIPs with $1$ or $2$ boundary lattice points.}
        \label{fig: PIPs with b=1 or b=2}
    \end{figure}
\end{example}

\begin{definition}
    A nonzero lattice vector $u = (\alpha_{1}, \dotsc, \alpha_{d}) \in
    \Z^{d}$ is \defing{primitive} if ${\gcd\setof{\alpha_{1}, \dotsc,
    \alpha_{d}} = 1}$, \emph{i.e.}, if there is no lattice point in
    the open segment $(0, u)$.
\end{definition}

\begin{definition}
    Let $X \subset \R^{d}$ be such that $\aff(X)$ is a hyperplane of
    the form $\aff(X) = \setof{a \in \R^{d} \st \inner{u, a} =
    \ldist{X}}$ for some nonzero primitive lattice vector $u \in
    \Z^{d}$ and $\ldist{X} \in \Rnn$.  Then $\ldist{X}$ is the
    \defing{lattice distance} from $X$ to the origin.  The lattice
    distance from $X$ to any point $p \in \R^{d}$ is defined to be the
    lattice distance from $X - p$ to the origin.
\end{definition}

Note that, for $X \subset \R^{d}$ such that $\aff(X)$ is a rational
hyperplane, the set $X$ is reticular if and only if $\ldist{X} \in
\Z$.

\begin{definition}
    The \defing{dual} (also called the \defing{polar}) of~$P$ is
    \begin{equation*}
        \dualp{P}
            \deftobe %
            \setof{%
                u \in \R^{d} %
                \st %
                \text{%
                    $\inner{u, a} \le 1$ for all $a \in P$%
                }%
            }.%
    \end{equation*}
\end{definition}

The dual $\dualp{P}$ is a polytope if and only if $\dim(P) = d$ and $0
\in \intr{P}$, in which case $\dualp{(\dualp{P})} = P$.  If moreover
$P$ is rational and $u_{F}$ is the primitive outer normal vector at
each facet $F$ of $P$, so that
\begin{equation*}
    P
        =
        \setof{%
            a \in \R^{d}
            \st
            \text{%
                $\inner{u_{F}, a} \le \ldist{F}$ for all facets $F$ of $P$
            }%
        },%
\end{equation*}
then
\begin{equation*}
    \dualp{P} 
        =%
        \conv\setof{
            \tfrac{1}{\ldist{F}} u_{F}
            \st 
            \text{$F$ is a facet of $P$}
        }.%
\end{equation*}
Thus, $\dualp{P}$ is integral if and only if $\ldist{F}^{-1} \in \Z$
for each facet $F$ of $P$.

\begin{definition}
    A \defing{reflexive} polytope is an integral polytope $P \subset
    \R^{d}$ such that $\dualp{P}$ is also an integral polytope.
\end{definition}
 
In particular, a full-dimensional integral polytope $P$ with $0 \in
\intr{P}$ is reflexive if and only if every facet of $P$ is at lattice
distance $1$ from the origin.  It is well-known that the reflexive
polygons in $\R^{2}$ are precisely the integral polygons in which the
only interior lattice point is the origin.

\section{Pseudoreflexive polygons}
\label{sec: pseudoreflexive polygons}

By analogy with reflexive polytopes, we call a PIP $P$ pseudoreflexive
if the dual of $P$ is an integral polytope.

\begin{definition}
    A \defing{pseudoreflexive} polytope, or a \defing{PRP} for short,
    is a PIP ${P \subset \R^{d}}$ such that $\dualp{P}$ is an integral
    polytope.
\end{definition}

In this section, we show that, similarly to the case of reflexive
polygons, the pseudoreflexive polygons are precisely the
pseudointegral polygons in which the only interior lattice point is
the origin.  This result recently appeared independently in
\cite[Corollary~5.6]{Boh2024preprint}.  We include a proof here for
completeness.

Let $P \subset \R^{d}$ be a polytope.  It will be convenient to
partition the boundary of $P$ into ``partially open facets''.  In
order to make this notion precise, let~$\F$ be the set of facets of
$P$, let $\E$ be the set of nonempty faces of $P$, and write $\Fi
\deftobe \setof{\intr{F} \st F \in \F}$, respectively $\Ei \deftobe
\setof{\intr{E} \st E \in \E}$, for the set of open facets,
respectively open faces, of~$P$.  (Note that a vertex of $P$ is its
own relative interior.)  A \defing{boundary-partition map} of~$P$ is a
map ${\pi \maps \Ei \to \Fi}$ such that ${\pi(\intr{E}) = \intr{F}}$
implies that $E$ is a face of $F$, for all $E \in \E$ and $F \in \F$.
(In particular, $\pi(F) = F$ for all $F \in \F$.)
Such a map $\pi$ induces a \defing{boundary partition} $\B_{\pi}
\deftobe \bigsetof{\bigcup_{\intr{E} \in \pi^{-1}(\intr{F})} \intr{E}
\st F \in \F}$ of $P$.  Thus, $\B_{\pi}$ is a set partition of
$\bnd{P}$ in which the closure of each part is a facet of $P$.  We
call the elements of $\B_{\pi}$ \defing{partially open facets} of $P$.

For example, if $P$ is an $n$-gon in $\R^{2}$ with vertices $v_{1},
\dotsc, v_{n}$ indexed counter\-clockwise with indices modulo $n$,
then the set $\setof{\ropen{v_{i}, v_{i+1}} \st i \in \Zmod{n}}$ of
half-open edges of $P$ is a boundary partition of $P$.  As another
example, if $P$ is a $d$-polytope, then a shelling order
$\parens{F_{1}, \dotsc, F_{n}}$ of the facets of $P$ yields the
boundary-partition map~$\pi$ according to which $\pi(\intr{E}) =
\intr{F_{i}}$ when $i = \min \setof{j \in [n] \st E \subset F_{j}}$.
(However, we do not require that our boundary partitions arise from
shellings in this way.)

It is a straightforward corollary of \cref{thm: Ehrhart's theorem} and
Ehrhart reciprocity that, if $P$ is a rational polytope with a
boundary partition $\B$, then the Ehrhart function $\ehr{F}$ of each
partially open facet $F \in \B$ is a degree\nobreakdash-$(\dim(P) -
1)$ quasipolynomial in which the leading coefficient $c_{\dim(F)}(t)$
equals $\relvol(F)$ when $tF$ is reticular and equals~$0$ otherwise.

\begin{proposition}
\label{prop: ever facet of a PIP is reticular}
    Let $P \subset \R^{d}$ be a $d$-dimensional rational PIP (or, more
    generally, a rational polytope such that $\bndp{tP} = t \bndp{P}$
    for all $t \in \Zp$).  Then every facet of $P$ is reticular,
    \emph{i.e.,} the lattice distance of every facet from the origin
    is an integer.
\end{proposition}

\begin{proof}
    Fix a boundary partition $\B$ of $P$.  Let $\RR \deftobe \setof{F
    \in \B \st \text{$F$ is reticular}}$ and $\RRbar \deftobe \B
    \setminus \RR$.  We must show that $\RRbar = \emptyset$.  Now, on
    the one hand, $\ehr{\bnd{P}} = \smash{\sum_{F \in \B} \ehr{F}}$
    because the partially open facets in $\B$ partition $\bnd{P}$.  On
    the other hand, for each $F \in \RRbar$, we have that $tF \cap
    \Z^{d} = \emptyset$ whenever $\den(\ldist{F}) \notdivides t \in
    \Zp$, so $\ehr{\bnd P}(t) = \smash{\sum_{F \in \RR} \ehr{F}(t)}$
    for each of the infinitely many values of $t \in \Zp$ such that
    $\den(\ldist{F}) \notdivides t$ for all $F \in \RRbar$.  In
    particular, $\sum_{F \in \B} \ehr{F}(t) = \ehr{\bnd{P}}(t) =
    \sum_{F \in \RR} \ehr{F}(t)$ for infinitely many values of $t \in
    \Zp$.  However, since $P$ is a PIP, we get from Ehrhart--Macdonald
    reciprocity that $\ehr{\bnd{P}} = \ehr{P} - \ehr{\intr{P}}$ is a
%     degree\nobreakdash-$(d - 1)$ 
    polynomial.  Thus, $\sum_{F \in \B} \ehr{F} = \sum_{F \in \RR}
    \ehr{F}$ (now as an equation of polynomials), and so $\sum_{F \in
    \RRbar} \ehr{F} = 0$.  Since the quasipolynomials $\ehr{F}$ all
    have nonnegative leading coefficient functions, it follows that
    $\RRbar = \emptyset$.
\end{proof}

\begin{proposition}[Proved independently in~\protect{\cite[Corollary~5.6]{Boh2024preprint}}]
\label{prop: characterization of PIPS with i = 1} %
    Let $P \subset \R^{2}$ be a $2$-dimensional PIP with $0 \in
    \intr{P}$.
%     (or, more generally, a $2$-dimensional polygon with $0 \in
%     \intr{P}$ such that $\bndp{tP} = t \bndp{P}$ for all $t \in \Zp$
%     and $P$ (but not necessarily $tP$ for integral $t \ge 2$)
%     satisfies Pick's formula)
    Then the following conditions are equivalent.
    \begin{enumeratetfae}
        \item  
        $\intr{P} \cap \Z^{2} = \setof{0}$.
        \label{cond: i=1}%
    
        \item  
        Every edge of $P$ is at lattice distance $1$ from the origin.
        \label{cond: edges are lattice distance 1}%
    
        \item  
        $\dualp{P}$ is integral (meaning that $P$ is a PRP).
        \label{cond: PRP}%
    \end{enumeratetfae}
\end{proposition}

Before we turn to the proof of \cref{prop: characterization of PIPS
with i = 1}, recall that the implications %
\ref{cond: i=1} %
    ~$\lif$~%
\ref{cond: edges are lattice distance 1}%
    ~$\liff$~%
\ref{cond: PRP} %
are standard results that hold for all rational full-dimensional
polytopes $P \subset \R^{d}$ with $0 \in \intr{P}$.  That is, neither
PIP-ness nor $2$-dimensionality is required for these implications.

\begin{sloppypar}
    Thus it remains only to prove that \ref{cond: i=1}~$\limp$~\ref{cond:
    edges are lattice distance 1}.  Note that this implication fails for
    non-PIPs such as the $4$-gon
    \begin{equation*}
        \setof{x \in \R^{2} \st \pm 2x \pm 3y \le 2}
            \quad = \quad
            \vcenter{\hbox{\scalebox{0.75}{\begin{tikzpicture}
                \drawlatticegrid{-1}{1}{-1}{1}{3}
                \drawfourgon{(1,0)}{(0,2/3)}{(-1,0)}{(0,-2/3)}
            \end{tikzpicture}}}}\;.
    \end{equation*}
    Moreover, the \ref{cond: i=1}~$\limp$~\ref{cond: edges are lattice
    distance 1} implication fails even in the integral case when ${\dim(P)
    \ge 3}$, since there exist non-reflexive lattice polytopes~$P$ with
    $\intr{P} \cap \Z^{d} = \setof{0}$ when $d \ge 3$.  (Note also that
    there exist non-PIPs that satisfy all of Conditions~\ref{cond:
    i=1}--\ref{cond: PRP}, such as the $8$-gon
    \begin{equation*}
        \setof{
            x \in \R^{2}
            \st
            \text{
                $\pm x \pm 2y \le 1$ and
                $\pm 2x \pm y \le 1$
            }
        }
            \quad = \quad %
            \vcenter{\hbox{\scalebox{0.75}{\begin{tikzpicture}
                \drawlatticegrid{-1}{1}{-1}{1}{6}
                \filldraw[polygonstyle] 
                    (1/2,0) --
                    (1/3,1/3) --
                    (0,1/2) --
                    (-1/3,1/3) --
                    (-1/2,0) --
                    (-1/3,-1/3) --
                    (0,-1/2) --
                    (1/3,-1/3) --
                    cycle;
            \end{tikzpicture}}}}\;,
    \end{equation*}
    which cannot be a PIP because there are no lattice points on its
    boundary.)
\end{sloppypar}

\begin{proof}[Proof of \cref{prop: characterization of PIPS with i = 1}] %
    It remains to prove that \ref{cond: i=1}~$\limp$~\ref{cond: edges
    are lattice distance 1}.  Let $P$ be a PIP such that $\intr{P}
    \cap \Z^{2} = \setof{0}$, and fix a boundary partition $\B$ of
    $P$.  We must show that $\ldist{E} = 1$ for all $E \in \B$.
    
    On the one hand, $P$, being a PIP, satisfies Pick's formula
    $\area{P} = \intp{P} + \frac{1}{2}\bndp{P} - 1$ by \cref{thm: PIP
    criteria}, so the area of~$P$ is $\frac{1}{2}\bndp{P}$.  Thus, the
    leading coefficient of $\ehr{P}$ is $\frac{1}{2}\bndp{P}$.
    
    On the other hand, for each partially open edge $E \in \B$, let
    $T_{E}$ be the partially open triangle defined by $T_{E} \deftobe
    \bigcup_{x \in E} \blopen{0, x}$.  Since $P \minusset{0}$ is a
    disjoint union of these partially open triangles, we have that
    $\ehr{P} = 1 + \sum_{E \in \B} \ehr{T_{E}}$.  Note that the
    leading coefficient of $\ehr{T_{E}}$ is $\area{T_{E}} =
    \frac{1}{2} \ldist{E} \llength{E}$ for all $E \in \B$, where
    $\llength{E}$ is the lattice length of $E$.  This last equation
    for the area of the triangle $T_{E}$ holds even though
    $\llength{E}$ is not a Euclidean length because we may use a
    lattice automorphism (which has determinant $1$) to map the edge
    $E$ to the line $\setof{(\xi, \ldist{E}) \st \xi \in \R}$, where
    $\llength{E}$ is now the Euclidean length of the image of $E$, and
    $\ldist{E}$ is the Euclidean height of the triangle with the image
    of $E$ as its base and $0$ as its apex.
    
    Hence, by comparing the leading coefficients of $\ehr{P}$ and $1 +
    \sum_{E \in \B} \ehr{T_{E}}$, we get that $\sum_{E \in \B}
    \ldist{E} \llength{E} = \bndp{P} = \sum_{E \in \B} \llength{E}$
    where second equation follows from Ehrhart--Macdonald reciprocity.
    Now, by \cref{prop: ever facet of a PIP is reticular}, $\ldist{E}
    \in \Zp$, and so $\ldist{E} \ge 1$, for all $E \in \B$.  It
    therefore follows from the equation $\sum_{E \in \B} \ldist{E}
    \llength{E} = \sum_{E \in \B} \llength{E}$ that $\ldist{E} = 1$
    for all $E \in \B$.
\end{proof}

\section{Counting boundary points of triangular PIPs}
\label{sec: reduction to number-theoretic lemma}

In this section, we derive a formula for the number $\bndp{P}$ of
lattice points on the boundary of a rational triangular PIP that
contains exactly one lattice point in its interior.  This formula
appears in \cref{prop: b for triangular PRPs} below.  Once this
formula is in hand, it will follow for purely number-theoretic reasons
given in \cref{sec: Number-theoretic lemmas for PRPs} that $\bndp{P}
\le 9$ and $\bndp{P} \ne 7$.

We first fix some additional notation.  Given vectors $u, v \in
\R^{2}$, we write $\prp{v}$ for the image of $v$ under
counterclockwise rotation about the origin by $\pi/2$; $\norm{v}$ for
the Euclidean length of $v$; and $\det(u, v)$ for the determinant of
the $2 \times 2$ matrix with columns $u$ and $v$, in that order.  Recall
that we write $\llength{E}$ for the lattice length of a line segment
$E$ that is parallel to a rational line.

The following proposition is probably not original.  However, we have
not found a reference in the literature.

\begin{proposition}\label{prop: edge vector of a polygon}
    Let $P = \setof{a \in \R^{2} \st \text{$\inner{u_{i}, a} \le
    \ldist{i}$ for all $i \in \Zmodn$}}$ be an $n$-gon, where the
    $u_{i}$ are outer edge normals of $P$ and are indexed
    counterclockwise about~$P$.  For each $i \in \Zmodn$, let~$E_{i}$
    be the edge of $P$ with outer normal $u_{i}$, and let
    $\setof{v_{i,i+1}} \deftobe E_{i} \cap E_{i+1}$.  Finally, for
    each $i, j \in \Zmodn$, let $\dett{ij} \deftobe \det(u_{i},
    u_{j})$.  
    
    Then $\dett{i-1, i},\, \dett{i, i+1} > 0$, and the edge~vector of
    the edge $E_{i}$ is
    \begin{equation}\label{eq: edge vector of a polygon}
        v_{i,i+1} - v_{i-1,i}
            =%
            \frac{
                \ldist{i-1} \, \dett{i, i+1}  
                - \ldist{i} \, \dett{i-1,i+1}
                + \ldist{i+1} \, \dett{i-1, i}
            }{
                \dett{i-1, i} \,
                \dett{i, i+1}
            }\;
            \prp{u_{i}}
    \end{equation}
    for all $i \in \Zmodn$.
\end{proposition}

\begin{proof} 
    (\emph{Sketch}) The positivity of $\dett{i, i+1}$ for $i \in
    \Zmodn$ corresponds to the condition that $P$ is convex.  To
    derive \cref{eq: edge vector of a polygon}, use Cramer's rule to
    solve for the vertex~$v_{i,i+1}$ in the linear system
    $\inner{u_{i}, v_{i,i+1}} = \ldist{i}$, $\inner{u_{i+1},
    v_{i,i+1}} = \ldist{i+1}$.  One finds that
    \begin{equation*}
        v_{i, i+1}
            =
            \frac{1}{\dett{i, i+1}} 
            (\ldist{i+1} \prp{u_{i}} - \ldist{i} \prp{u_{i+1}})
    \end{equation*}
    for $i \in \Zmodn$.  Now take the difference $v_{i, i+1} - v_{i-1,
    i}$ and simplify the result by again using Cramer's rule to show
    that $\dett{i, i+1} u_{i-1} + \dett{i-1, i} u_{i+1} = \dett{i-1,
    i+1} u_{i}$.
\end{proof}

\begin{corollary}\label{cor: lattice length of an edge in an n-gon}
    Let $P \subset \R^{2}$ be a rational $n$-gon (or, more generally,
    an $n$-gon in which every edge has a rational outer normal).
    Adopt the notation of \cref{prop: edge vector of a polygon}, with
    the additional condition that $u_{i}$ be the \emph{primitive}
    outer normal of the edge $E_{i}$ for $i \in \Zmodn$.  Then the
    lattice length of the edge $E_{i}$ is
    \begin{equation}\label{eq: lattice length of an edge in an n-gon}
        \llength{E_{i}}
            =%
            \frac{
                \ldist{i-1} \, \dett{i, i+1}  
                - \ldist{i} \, \dett{i-1,i+1}
                + \ldist{i+1} \, \dett{i-1, i}
            }{
                \dett{i-1, i} \,
                \dett{i, i+1}
            }
    \end{equation}
    for $i \in \Zmod{n}$.
\end{corollary}

\begin{proof}
    Since $u_{i}$ is the primitive outer normal to $E_{i}$, the vector
    $\prp{u_{i}}$ is a primitive lattice vector that is parallel to
    $E_{i}$.  The lattice length of the edge $E_{i}$ is thus the
    Euclidian length of $E_{i}$ divided by the Euclidean length of
    $\prp{u_{i}}$.  Therefore, by \cref{eq: edge vector of a polygon},
    \begin{equation*}
        \llength{E_{i}}
            =
            \frac{\norm{v_{i,i+1} - v_{i-1,i}}}{\norm{\prp{u_{i}}}}
            =
            \frac{
                \ldist{i-1} \, \dett{i, i+1}  
                - \ldist{i} \, \dett{i-1,i+1}
                + \ldist{i+1} \, \dett{i-1, i}
            }{
                \dett{i-1, i} \,
                \dett{i, i+1}
            }. \qedhere
    \end{equation*}
\end{proof}

Note that the determinants $\dett{ij} = \det(u_{i}, u_{j})$ appearing in
\cref{eq: lattice length of an edge in an n-gon} are all integers.
Moreover, the determinants of the form $\dett{i, i+1}$ are
\emph{positive} integers because $P$ is convex and the edges are
indexed counterclockwise about $P$.

Unfortunately, no such guarantee of positivity is possible for the
determinants of the form $\dett{i-1,i+1}$ in general.  However, in the
case in which $P$ is a rational \emph{triangle}, \cref{eq: lattice
length of an edge in an n-gon} simplifies nicely.  For then, computing
modulo $3$ in the subscripts, we have that ${-\dett{i-1, i+1} =
\dett{i+1, i+2} \in \Zp}$.

This control over the signs of all of the determinants $\dett{ij}$ in
the ${n = 3}$ case is what permits us to apply the number-theoretic
bounds that we find below to pseudointegral triangles.  For $n$-gons
with $n \ge 4$, we lose this control.  Therefore, our methods do not
immediately resolve whether every $n$-gonal PRP $P$ with $n \ge 4$
satisfies $\intp{P} \le 9$.

\begin{corollary}\label{cor: lattice length of an edge in a triangle}
    Let $T$ be a rational triangle (or, more generally, an triangle in
    which every edge has a rational outer normal) and let $E$ be an
    edge of $T$.  Then the lattice length of $E$ is
    \begin{align}\label{eq: lattice length of an edge in a triangle}
        \llength{E}
            &=
            \frac{
                \alpha x + \beta y + \gamma z
            }{xyz} \;
            x,
    \end{align}
    where $T = \setof{a \in \R^{2} \st \inner{u, a} \le \alpha, \,
    \inner{v, a} \le \beta, \,\inner{w, a} \le \gamma}$; $u$ is the
    primitive outer normal of $E$; $v$ and $w$ are the respective
    primitive outer normals of the remaining edges of $T$ in
    counterclockwise order from $E$; and $x \deftobe \det(v, w)$, $y
    \deftobe \det(w, u)$, and $z \deftobe \det(u, v)$.
\end{corollary}

(The construction of $(x,y,z) \in \Zp^{3}$ from the rational triangle
$T$ and a fixed edge $E$ of $T$ very nearly defines an invariant of
rational triangles modulo the automorphism group of the lattice.  We
do not yet have an invariant because of the dependence on the choice
of $E$.  We complete the construction of an invariant $T \mapsto
(x,y,z)$ in \cref{defn: an invariant for rational triangles} below.)

Thus far, the formulas in \cref{prop: edge vector of a polygon} and
its corollaries hold for all rational triangles (at least).  We now
introduce the pseudointegrality condition.  The following proposition
provides a formula for $\bndp{T}$ in the case in which the triangle
$T$ is pseudointegral.

\begin{proposition}\label{prop: i A and b for triangular PIPs} 
    Let $T \subset \R^{2}$ be a rational triangular PIP,
    % [requires only that \bndp{T} is a PIP.]
    and adopt the notation of \cref{cor: lattice length of an edge in
    a triangle}.  Then the number of lattice points on the boundary of
    $T$ is
    \begin{equation*}
        \bndp{T}
            =%
            \frac{
                \parens{
                    \alpha x 
                    + \beta y
                    + \gamma z
                }
                \parens{
                    x
                    + y
                    + z
                }
            }{xyz}.
    \end{equation*}
\end{proposition}

\begin{proof}
    Since $T$ is a PIP, $\bndp{T}$ is the sum of the lattice lengths
    of the edges of $T$.  The result thus follows from \cref{cor:
    lattice length of an edge in a triangle}.
\end{proof}

In the case in which $\intp{T} = 1$ and $T$ is pseudointegral,
% [actually need just that $T$ is Pick and that $\bnd{P}$ is 
% PIPal]
the previous proposition yields the following as a corollary.

\begin{proposition}\label{prop: b for triangular PRPs} 
    Let $T$ be a triangular PIP with $\intp{T} = 1$.  Let $x, y, z \in
    \Zp$ be as defined in \cref{cor: lattice length of an edge in a
    triangle}.  Then the number of lattice points on the boundary of
    $T$ is
    \begin{equation}\label{eq: b for triangular PRPs}
        \bndp{T}
            =
            \frac{
                \parens{
                    x
                    + y
                    + z
                }^{2}
            }{xyz}.
    \end{equation}
\end{proposition}

\begin{proof}
    Since we may apply a lattice translation to $T$ without changing
    the number of lattice points on the boundary or in the interior of
    $T$, we may suppose without loss of generality that $\intr{T} \cap
    \Z^{2} = \setof{0}$.  Thus, $\alpha = \beta = \gamma = 1$ in
    \cref{prop: characterization of PIPS with i = 1}.  The result now
    follows from \cref{prop: i A and b for triangular PIPs}.
\end{proof}

Again, the values $x, y, z$ in \cref{eq: b for triangular PRPs} are
positive integers because $T$ is rational and convex (see the
paragraph following the proof of \cref{cor: lattice length of an edge
in an n-gon}).  Of~course, the left-hand side $\bndp{T}$ is also an
integer.  In the next section, we prove that, for \emph{any} three
positive integers $x, y, z$, if $(x+y+z)^{2}/(xyz)$ is also an
integer, then that integer is at most $9$ and not equal to $7$.  This
will prove \cref{thm: first main result}.

\section{Number-theoretic lemmas} %
\label{sec: Number-theoretic lemmas for PRPs}

In this section, we prove \cref{lem: number-theory lemma for
triangular PRPs}, a statement of elementary number theory from which
our first main result (\cref{thm: first main result}) follows.  We
restate \cref{lem: number-theory lemma for triangular PRPs} here for
the convenience of the reader.

\begin{lemma}
\label{lem: number-theory lemma for triangular PRPs - restated}%
    Let $x, y, z \in \Zp$ and let $b \deftobe
    \frac{(x+y+z)^{2}}{xyz}$.  If $b \in \Z$, then $b \le 9$ and $b
    \ne 7$.
\end{lemma}

(See \cref{thm: number-theory lemma for general n} below for a
generalization of \cref{lem: number-theory lemma for triangular PRPs -
restated} to $n \ge 2$ integers $x_{1}, \dotsc, x_{n}$.  However, we
relegate this result to an appendix because we do not know of an
application to PIPs.)

Before proving \cref{lem: number-theory lemma for triangular PRPs -
restated}, we give an interesting example involving the Fibonacci
sequence.  In the next section, we will show how this example gives
rise to the infinite family of PIPs in \cref{exm: Fibonacci PRPs} from
the introduction.

\begin{example}\label{exm: Fibonacci solution}
    Let $F_{j}$ denote the $j$th Fibonacci number, where $F_{1} =
    F_{2} = 1$.  Then $(x,y,z) \deftobe (1, F_{2j-1}^{2},
    F_{2j+1}^{2})$ satisfies $\frac{(x+y+z)^2}{xyz} = 9$ for all $j
    \in \Zp$.  One may verify this equation directly by using a
    special case of Cassini's identity for the Fibonacci sequence: For
    all $j \in \Zp$,
    \begin{equation*}
        F_{2j-1} F_{2j+1} - F_{2j}^{2} = 1.
    \end{equation*}
    We obtain $1+(F_{2j+1}-F_{2j-1})^{2} = F_{2j-1} F_{2j+1}$, and
    expanding yields $ 1 + F_{2j-1}^{2} + F_{2j+1}^{2} = 3 F_{2j-1}
    F_{2j+1}$, thus $(1 + F_{2j-1}^{2} + F_{2j+1}^{2})^{2} = 9
    F_{2j-1}^{2} F_{2j+1}^{2}$, as desired.
\end{example}

Note that \cref{lem: number-theory lemma for triangular PRPs - restated} is
fundamentally a number-theoretic result.  No such upper bound holds if
$x, y, z$ are merely assumed to be positive integers without the
divisibility condition, since $\lim_{z \to \infty} \frac{(x + y +
z)^{2}}{x y z} = \infty$ when $x$ and $y$ are fixed.
% For example,
% \begin{equation*}
%     \frac{(2 + 2 + 32)^{2}}{2 \cdot 2 \cdot 32} = \frac{81}{8} = 10.125
%     \qquad \text{and} \qquad
%     \frac{(3+4+245)^{2}}{3 \cdot 4 \cdot 245} = \frac{108}{5} =  21.6.
% \end{equation*}
Furthermore, even with this divisibility condition, the positivity of
the integers $x, y, z$ cannot be omitted since, for example, $\frac{(1
- 1 + z)^{2}}{1 (-1) (-z)} = z$ and $\frac{((-1) + (1-z) +
z^{2})^{2}}{(-1)(1-z)z^{2}} = z-1$ for all $z \in \Zp$.

Our proof of \cref{lem: number-theory lemma for triangular PRPs -
restated} uses a remarkable technique in number theory that was
evidently introduced by Hurwitz c.~1907 \cite{Hur1963} and which is
now often called ``\mbox{Vieta jumping}''.  With this technique, we
can reduce a given solution ${(x, y, z) \in \Zp}$ to the equation
$\frac{(x+y+z)^{2}}{xyz} = b$ (where $b \in \Z$ is fixed) to a
``Vieta-reduced'' \mbox{solution}, which in our present situation we
define as follows.

\begin{definition}
    Fix $b \in \Zp$.  A \defing{solution} is an element of the set
    \begin{equation*}
        S_{b} 
            \deftobe %
            \setof{(x, y, z) \in \Zp^{3} \st (x + y + z)^{2} = b x y  z}.
    \end{equation*}
    A solution $(x, y, z) \in S_{b}$ is \defing{Vieta reduced} if $x
    \le y \le z \le x + y$.
\end{definition}

(\Cref{thm:reduced} below lists all Vieta-reduced solutions.)

\begin{lemma}\label{lem: a solution implies a Vieta-reduced solution}
    Fix $b \in \Zp$.  If $S_{b}$ is nonempty, then $S_{b}$ contains a
    Vieta-reduced solution.
\end{lemma}

\begin{proof}
    Let $(x, y, z) \in S_{b}$.  Without loss of generality, assume
    that $x \le y \le z$.  If $z \le x + y$, then we are done, so
    suppose that $z > x + y$.  Since $(x, y, z)$ is a solution, we
    have that $z$ is a root of the polynomial $f(t) \in \Z[t]$ defined
    by
    \begin{equation*}
        f(t) 
            \deftobe  
            t^{2}
            - \bigparens{bxy - 2(x+y)} t
            + (x + y)^{2}.
    \end{equation*} 
    Letting $z'$ be the other root of $f(t)$, we get from $z + z' =
    bxy - 2(x+y)$ that $z' \in \Z$ and from $z z' = (x + y)^{2} > 0$
    that $z' \ge 1$.  Thus we have a solution $(x, y, z') \in S_{b}$
    with $z' < x + y < z$.  If moreover $y \le z'$, then we are done.
    Otherwise, $z' < y$, so we reorder the solution $(x, y, z')$ in
    increasing order, getting a solution $(x_{1}, y_{1}, z_{1}) \in
    S_{b}$ with $1 \le x_{1} \le y_{1} \le z_{1}$ and $z_{1} = y < z$.
    
    If $z_{1} \le x_{1} + y_{1}$, then the solution $(x_{1}, y_{1},
    z_{1})$ is Vieta-reduced, and so we are done.  Otherwise, we
    iteratively apply the reduction procedure in the previous
    paragraph to our current solution.  At the conclusion of each
    iteration, we have a solution in~$S_{b}$ in which the
    ``$z$-value'' is a positive integer that is strictly less than the
    $z$-value of the solution with which the iteration began.
    Therefore, we eventually arrive at a Vieta-reduced solution
    in~$S_{b}$.
\end{proof}

We are now ready to prove \cref{lem: number-theory lemma for
triangular PRPs - restated}.

\begin{proof}[Proof of \cref{lem: number-theory lemma for triangular PRPs - restated}]
    By \cref{lem: a solution implies a Vieta-reduced solution}, we may
    suppose without loss of generality that our solution $(x, y, z)
    \in S_{b}$ is Vieta reduced, so that $1 \le x \le y \le z \le x +
    y$.  Hence,
    \begin{align*}
        b
            &=
            \frac{(x + y + z)^{2}}{xyz}
            =
            \frac{x + y + z}{xy} \cdot \frac{x + y + z}{z} \\ 
            &\le%
            \frac{2(x + y)}{xy} \cdot \frac{x + y + z}{z}
            =%
            2\parens{
                \frac{1}{y} + \frac{1}{x}
            }
            \parens{\frac{x}{z} + \frac{y}{z} + 1}.
    \end{align*}
    If $x \ge 2$, then $\frac{1}{y} \le \frac{1}{x} \le \frac{1}{2}$
    and $\frac{x}{z}, \frac{y}{z} \le 1$, so $b \le 6$.  Furthermore,
    if $x = 1$ and $y \ge 3$, then $\frac{x}{z} \le \frac{1}{y} \le
    \frac{1}{3}$ and $\frac{y}{z} \le 1$, so $b \le \frac{56}{9} < 7$,
    which, since $b \in \Z$, implies again that $b \le 6$.
    
    It remains to consider the cases in which $(x, y) \in
    \setof{(1,1), (1,2)}$.  On the one hand, if $(x, y) = (1, 1)$,
    then we get from $z \le x + y$ that $z \in \setof{1,2}$.  If $z =
    1$, then $b = 9$, while if $z = 2$, then $b = 8$.  On the other
    hand, if $(x,y) = (1,2)$, then we get from $y \le z \le x + y$
    that $z \in \setof{2,3}$.  Indeed, since $b \in \Z$, we have that
    $z = 3$ and hence $b = 6$.
    
    Thus we have shown that $1 \le b \le 9$ and $b \ne 7$, as claimed.
\end{proof}

Our first main result (\cref{thm: first main result}) now follows
immediately.

\begin{proof}[Proof of \cref{thm: first main result}]
    Let $T$ be a rational triangular PIP. By \cref{prop: b for
    triangular PRPs}, there exist $x, y, z \in \Zp$ such that
    $\bndp{T} = (x + y + z)^{2}/(xyz)$.  Therefore, by \cref{lem:
    number-theory lemma for triangular PRPs - restated}, $1 \le
    \bndp{T} \le 9$ and $\bndp{T} \ne 7$.
\end{proof}

We conclude this section by finding all Vieta-reduced solutions to the
equation $b = (x + y + z)^{2}/(xyz)$.  We will use this result in the
next section to construct infinitely many inequivalent triangular PIPs
$T$ such that $\intp{T} = 1$ and $\bndp{T} = b$.

\begin{theorem}\label{thm:reduced}
    Fix $b \in \Zp$, and let $(x, y, z)$ be a Vieta-reduced solution
    in $S_{b}$.  Then $(b,x,y,z)$ is one of the following thirteen
    $4$-tuples:
    \begin{center}
        \begin{tabular}{CCCCC||CCCCCC||CCCCC}
            b & x &  y &  z &&& b & x & y & z &&& b & x & y & z \\
            \hline
            1 & 5 & 20 & 25 &&& 2 & 3 & 6 & 9 &&& 4 & 2 & 2 & 4 \\
            1 & 6 & 12 & 18 &&& 2 & 4 & 4 & 8 &&& 5 & 1 & 4 & 5 \\
            1 & 8 &  8 & 16 &&& 3 & 2 & 4 & 6 &&& 6 & 1 & 2 & 3 \\
            1 & 9 &  9 &  9 &&& 3 & 3 & 3 & 3 &&& 8 & 1 & 1 & 2 \\
              &   &    &    &&&   &   &   &   &&& 9 & 1 & 1 & 1
        \end{tabular}
    \end{center}
\end{theorem}

\begin{proof}
    Using the inequalities $x \leq y \leq z \leq x + y \leq 2y$, we
    obtain $16y^2 \geq (x + y + 2y)^2 \geq (x + y + z)^2 = b x y z
    \geq bxy^2$, thus $1 \leq x \leq \frac{16}{b}$.  Let $w \deftobe
    z-y$.  Then $w$ satisfies $0 \leq w \leq x$.  For each choice of
    $w$ and $x$, the equation $(x+y+w+y)^2 = bxy(w+y)$ is quadratic in
    $y$, and so we may solve for $y$.  We then iterate over all $251$
    choices of integers $b,x,w$ with $1 \le b \le 9$, $b \ne 7$, $1
    \leq x \leq \frac{16}{b}$, $0 \leq w \leq x$, and obtain the
    complete list of Vieta-reduced solutions given above.
\end{proof}

Every positive-integer solution to $b = (x+y+z)^2/(xyz)$ can be found
by taking one of the Vieta-reduced solutions in \cref{thm:reduced},
reordering it, and then performing a sequence of ``Vieta jumps'',
reversing the reduction process in the proof of \cref{lem: a solution
implies a Vieta-reduced solution}.  These solutions are therefore the
nodes of a graph in which the edges correspond to Vieta jumps.  It
follows from \cref{thm:reduced} that this graph is a $3$-regular
rooted forest with~$51$ connected components falling into $13$
equivalence class modulo reordering, corresponding to the $51$
distinct reorderings of the $13$ solutions $(x,y,z)$ in
\cref{thm:reduced}.

\section{Infinite families of pseudoreflexive triangles}
\label{sec: Infinite families of pseudo-reflexive triangles}

Recall that we write $S_{b}$ for the set of all solutions $(x,y,z) \in
\Zp^{3}$ to the equation $b = \smash{(x + y + z)^{2}/(xyz)}$.  For
each Vieta-reduced solution in $S_{b}$, we will generate a particular
infinite family of solutions.  We will then use this family to
construct infinitely many triangular PRPs with $b$ lattice points on
their boundaries that are pairwise inequivalent modulo automorphisms
of the lattice, answering a question posed by Bohnert
\cite[Remark~5.9]{Boh2024preprint}.

As indicated at the end of \cref{sec: Number-theoretic lemmas for
PRPs}, each Vieta-reduced solution is the root of a rooted $3$-regular
tree of solutions.  The procedure that we give below for constructing
a PIP from a solution does not work for every solution in this tree.
However, by specializing how the Vieta jumps are performed, one can
generate an infinite family of solutions that are ``well-behaved'' for
our purposes, meaning that our procedure will construct a PIP from
each of the solutions in this family.

\begin{example}\label{exm: Fibonacci solution via Vieta}
    The family of solutions $(1, F_{2j-1}^{2}, F_{2j+1}^{2})$ for the
    case $b = 9$ from \cref{exm: Fibonacci solution} is generated by
    starting with the Vieta-reduced solution $(1, 1, 1)$ for $b = 9$
    and then successively Vieta jumping the second-largest value of
    $(x, y, z)$.
\end{example}

The ``Vieta-jumping recipe'' used in \cref{exm: Fibonacci solution via
Vieta} may be applied to any Vieta-reduced solution, which we do now
with a view towards constructing a corresponding infinite families of
PIPs.

For the remainder of this section, fix a positive integer~$b$ such
that $1 \le b \le 9$ and $b \ne 7$, as well as a Vieta-reduced
solution $s_{0} \in S_{b}$.  Thus,~$s_{0}$ is one of the solutions
listed in \cref{thm:reduced}.  Set $s_{0} \betodef (x, y_{0}, z_{0})$,
and recursively define $s_{j} = (x, y_{j}, z_{j})$ for all~$j \in \Zp$
by setting $y_{j+1} \deftobe z_{j}$
and %
$
    z_{j+1} 
        \deftobe %
        b x y_{j+1} - 2(x + y_{j+1}) - y_{j}
$ %
for all~$j \in \Znn$.

\begin{lemma}\label{lem: Infinitely many solutions}
    The triple $s_{j} = (x, y_{j}, z_{j})$ is a solution in~$S_{b}$
    with $z_{j-1} < z_{j}$ for all $j \in \Zp$.  Moreover, $x \divides
    y_{j}, z_{j}$ for all $j \in \Znn$.
\end{lemma}

\begin{remark}
    The divisibility conditions satisfied by the solutions in
    \cref{lem: Infinitely many solutions} are not satisfied by all
    solutions.  An example for which these divisibility conditions
    fail is $(4,5,81) \in S_{5}$.  In particular, the procedure given
    below for constructing PIPs from certain solutions in $S_{5}$
    cannot be applied directly to the solution $(4,5,81)$.
\end{remark}

\begin{proof}[Proof of \cref{lem: Infinitely many solutions}]
    \begin{sloppypar}
        That $s_{j}$ is a solution satisfying $z_{j-1} < z_{j}$ holds in
        the base case $j = 1$ by inspection of the Vieta-reduced solutions
        in \cref{thm:reduced}.  Proceeding by induction, let $j \in \Zp$
        be such that $(x, y_{j}, z_{j})$ is a solution with $z_{j-1} <
        z_{j}$.  Then~${(x, y_{j+1}, y_{j}) = (x, z_{j}, y_{j})}$ is a
        solution with $y_{j} = z_{j-1} < z_{j} < x + z_{j} = x + y_{j+1}$.
        Thus, $y_{j}$ is the smaller of the two roots of the polynomial
        \begin{equation*}
            t^{2}
            - \bigparens{b x y_{j+1} - 2(x + y_{j+1})} t
            + (x + y_{j+1})^{2}.
        \end{equation*} 
        Hence, $z_{j+1}$ is the larger of the roots of this polynomial,
        which means that $(x, y_{j+1}, z_{j+1})$ is a solution with $z_{j}
        < x + y_{j+1} < z_{j+1}$.
    \end{sloppypar}
    
    That $s_{j}$ satisfies the divisibility conditions $x \divides
    y_{j}, z_{j}$ again holds in the base case~${j = 0}$ by inspection
    of the Vieta-reduced solutions.  Let $j \in \Znn$ be such that ${x
    \divides y_{j}, z_{j}}$.  Then $x \divides z_{j} = y_{j+1}$ and $x
    \divides b x z_{j} - 2(x + z_{j}) - y_{j} = b x y_{j+1} - 2(x +
    y_{j+1}) - y_{j} = z_{j+1}$, and the claim is proved.
\end{proof}

Let $\SSS \deftobe \setof{s_{j} \st j \in \Znn}$ be the family of
solutions in $S_{b}$ generated by our chosen Vieta-reduced solution
$s_{0}$ as described above.  Note that $\SSS$ is infinite because the
``$z$-values'' are strictly increasing by \cref{lem: Infinitely many
solutions}.  For each $(x,y,z) \in \SSS$, let $T_{xyz}$ be the
triangle defined by $T_{xyz} \deftobe \setof{a \in \R^{2} \st
\text{$\inner{u_{i}, a} \le 1$ for $i = 1,2,3$}}$, where
\begin{align}\label[eqns]{eqns: outer normals}
    u_{1} 
        & \deftobe %
        \begin{bmatrix}
            y \\
            \tfrac{y + z}{x}
        \end{bmatrix},
    &
    u_{2} 
        & \deftobe %
        \begin{bmatrix}
            -x \\
            -1
        \end{bmatrix},
    &
    u_{3} 
        & \deftobe %
        \begin{bmatrix*}[r]
             0 \\
            -1
        \end{bmatrix*}.
\end{align}

\begin{theorem}\label{thm: Infinitely many PRPs}
    For every solution $(x,y,z) \in \SSS$, the triangle $T_{xyz}$
    defined above is a PIP with $\intp{T} = 1$ and $\bndp{T} = b$.
\end{theorem}

\begin{example}
    The triangle $T_{j}$ in \cref{exm: Fibonacci PRPs} is the PIP
    $T_{xyz}$ in the case where $(x, y, z) \deftobe (1, F_{2j-1}^{2},
    F_{2j+1}^{2})$ is the solution discussed in \cref{exm: Fibonacci
    solution,exm: Fibonacci solution via Vieta}.
\end{example}

While reading the proof of \cref{thm: Infinitely many PRPs}, recall
that the matrix
\begin{equation*}
    \begin{bmatrix}
        1 - rs & r^{2}  \\
        -s^{2} & 1 + rs
    \end{bmatrix}
\end{equation*}
defines a linear transformation that shears the plane parallel to the
vector $(r,s)$ for all nonzero $(r, s) \in \R^{2}$.  Moreover, if the
entries of the matrix are all integers, then this linear
transformation restricts to an automorphism of $\Z^{2}$ and therefore
preserves Ehrhart quasipolynomials.

\begin{proof}[Proof of \cref{thm: Infinitely many PRPs}]
    We will show that the Ehrhart quasipolynomial of $T \deftobe
    T_{xyz}$ equals the Ehrhart polynomial of the partially-open
    integral triangle
    \begin{equation*}
        B 
            \deftobe %
            \bigparens{
                \conv\setof{(0,0), (0, -1), (b, -1)}
                \setminus
                \conv\setof{(0,0), (b, -1)}
            }
            \cup
            \setof{(0,0)}.
    \end{equation*}
    This partially open triangle has Ehrhart polynomial $\ehr{B}(t) =
    \parens{\tfrac{b}{2} t^{2} + \tfrac{b+2}{2}t + 1} - \parens{t + 1}
    + 1 = \tfrac{b}{2} t^{2} + \tfrac{b}{2}t + 1$.  Once we have shown
    that $\ehr{T} = \ehr{B}$, it will then follow that $T$ is a PIP,
    and moreover that $\intp{T} = 1$ and $\bndp{T} = b$ by \cref{thm:
    PIP criteria}.

    To this end, let $E_{i}$ be the edge of $T$ that is perpendicular
    to $u_{i}$, and let $v_{i,i+1}$ be the vertex at the intersection
    of $E_{i}$ and $E_{i+1}$, for all $i \in \Zmod{3}$.  One finds
    that
    \begin{align}\label[eqns]{eqns: vertices of infinitely many triangular PRPs}
        v_{12} 
            & = %
            \begin{bmatrix*}[c]
                -\tfrac{1}{xz}\parens{x + y + z} \\
                \tfrac{1}{z}(x + y)
            \end{bmatrix*},
        &
        v_{23} 
            & = %
            \begin{bmatrix*}[r]
                0 \\
                -1
            \end{bmatrix*},
        &
        v_{31} 
            & = %
            \begin{bmatrix*}[c]
                \tfrac{1}{xy}\parens{x + y + z} \\
                -1
            \end{bmatrix*}.
    \end{align}
    
    The triangle $T$ decomposes as a disjoint union of partially-open
    triangles $T_{i}$ over the edges $E_{i}$.  More precisely, let
    \begin{align*}
        T_{1} 
            & \deftobe %
            \conv \setof{0, v_{12}, v_{23}} 
            \setminus
            \conv \setof{0, v_{12}}, \\
        T_{2} 
            & \deftobe %
            \conv \setof{0, v_{23}, v_{31}}
            \setminus
            \conv \setof{0, v_{23}}, \\
        T_{3} 
            & \deftobe %
            \bigparens{%
                \conv \setof{0, v_{31}, v_{12}}
                \setminus
                \conv \setof{0, v_{31}}
            }%
            \cup
            \setof{0}.
    \end{align*}
    Thus, $T = T_{1} \sqcup T_{2} \sqcup T_{3}$ as a disjoint union.
    We will prove that $\ehr{T} = \ehr{B}$ by showing that $B$ is a
    disjoint union of partially-open triangles that are images of
    $T_{1}, T_{2}, T_{3}$, respectively, under automorphisms of the
    integer lattice.
    
    To assist in following the argument, we will take as a running
    example the Vieta-reduced solution $(3,6,9) \in S_{2}$.  In this
    case,
    \begin{equation*}
        T_{1}
        \sqcup
        T_{2}
        \sqcup
        T_{3}
            \quad = \quad
            \vcenter{\hbox{\scalebox{1}{\begin{tikzpicture}
                \drawlatticegrid{-1}{2}{-1}{1}{3}
                
                \fill[fill=black!25]
                    (0, 0) -- (1, -1) -- (-2/3, 1) -- cycle;
                \begin{pgfonlayer}{polygon1}
                    \draw [thick] (0, 0) -- (1, -1);
                    \draw [thick] (1, -1) -- (-2/3, 1);
                \end{pgfonlayer}

                \fill[fill=black!25]
                    (0, 0) -- (-2/3, 1) -- (0, -1) -- cycle;
                \begin{pgfonlayer}{polygon1}
                    \draw [thick] (0,0) -- (-2/3, 1);
                    \draw [thick] (-2/3, 1) -- (0, -1);
                \end{pgfonlayer}

                \fill[fill=black!25]
                    (0, 0) -- (0, -1) -- (1, -1) -- cycle;
                \begin{pgfonlayer}{polygon1}
                    \draw [thick] (0, 0) -- (0, -1);
                    \draw [thick] (0, -1) -- (1, -1);
                \end{pgfonlayer}
            \end{tikzpicture}}}}\;.
    \end{equation*}

    Let
    \begin{align*}
        U_{23}
            & \deftobe %
            \begin{bmatrix*}[c]
                1 & 0 \\
                x & 1
            \end{bmatrix*},
        &
        U_{31}
            \deftobe
            \begin{bmatrix*}[c]
                 1 + \tfrac{x + y + z}{x} & \tfrac{(x+y+z)^{2}}{x^{2}y} \\
                -y & 1 - \tfrac{x + y + z}{x}
            \end{bmatrix*}.
    \end{align*}
    Then, for $i \in \setof{2, 3}$, the matrix $U_{i, i+1}$
    corresponds to a linear transformation that shears the plane
    parallel to the vector $v_{i, i+1}$ and maps the vertex $v_{12}$
    into the horizontal line through $(0,-1)$, which contains the edge
    $E_{3}$.  Moreover, $U_{i, i+1}$ corresponds to an automorphism of
    the integer lattice because $(x,y,z)$ satisfies the divisibility
    conditions in \cref{lem: Infinitely many solutions}.  Let $T_{1}'
    \deftobe U_{31}T_{1}$ and $T_{2}' \deftobe U_{23}T_{2}$.  Define
    $T' \deftobe T_{1}' \cup T_{2}' \cup T_{3}$.  Then $T' = T_{1}'
    \sqcup T_{2}' \sqcup T_{3}$ as a disjoint union.  It follows that
    $\ehr{T} = \ehr{T'}$.
    
    In our running example, applying $U_{31}$ followed by $U_{23}$
    yields
    \begin{equation*}
        T_{1}'
        \sqcup
        T_{2}
        \sqcup
        T_{3}
            \quad = \quad
            \vcenter{\hbox{\scalebox{1}{\begin{tikzpicture}
                \drawlatticegrid{-1}{2}{-1}{1}{3}
                
                \fill[fill=black!25]
                    (0, 0) -- (1, -1) -- (4/3, -1) -- cycle;
                \begin{pgfonlayer}{polygon1}
                    \draw [thick] (0, 0) -- (1, -1);
                    \draw [thick] (1, -1) -- (4/3, -1);
                    \draw [thick, dashed] (0, 0) -- (4/3, -1);
                \end{pgfonlayer}

                \fill[fill=black!25]
                    (0, 0) -- (-2/3, 1) -- (0, -1) -- cycle;
                \begin{pgfonlayer}{polygon1}
                    \draw [thick] (0,0) -- (-2/3, 1);
                    \draw [thick] (-2/3, 1) -- (0, -1);
                \end{pgfonlayer}

                \fill[fill=black!25]
                    (0, 0) -- (0, -1) -- (1, -1) -- cycle;
                \begin{pgfonlayer}{polygon1}
                    \draw [thick] (0, 0) -- (0, -1);
                    \draw [thick] (0, -1) -- (1, -1);
                \end{pgfonlayer}
            \end{tikzpicture}}}}\;.
    \end{equation*}
    followed by
    \begin{equation*}
        T_{1}'
        \sqcup
        T_{2}'
        \sqcup
        T_{3}
            \quad = \quad
            \vcenter{\hbox{\scalebox{1}{\begin{tikzpicture}
                \drawlatticegrid{-1}{2}{-1}{1}{3}
                
                \fill[fill=black!25]
                    (0, 0) -- (1, -1) -- (4/3, -1) -- cycle;
                \begin{pgfonlayer}{polygon1}
                    \draw [thick] (0, 0) -- (1, -1);
                    \draw [thick] (1, -1) -- (4/3, -1);
                    \draw [thick, dashed] (0, 0) -- (4/3, -1);
                \end{pgfonlayer}

                \fill[fill=black!25]
                    (0, 0) -- (-2/3, -1) -- (0, -1) -- cycle;
                \begin{pgfonlayer}{polygon1}
                    \draw [thick] (0,0) -- (-2/3, -1);
                    \draw [thick] (-2/3, -1) -- (0, -1);
                \end{pgfonlayer}

                \fill[fill=black!25]
                    (0, 0) -- (0, -1) -- (1, -1) -- cycle;
                \begin{pgfonlayer}{polygon1}
                    \draw [thick] (0, 0) -- (0, -1);
                    \draw [thick] (0, -1) -- (1, -1);
                \end{pgfonlayer}
            \end{tikzpicture}}}}\;.
    \end{equation*}

    Finally, let $T_{2}'' \deftobe \rinlinematrix{1}{-b}{0}{1}
    T_{2}'$.  Then $B = T_{1}' \sqcup T_{2}'' \sqcup T_{3}$ as a
    disjoint union.  In our running example, we have
    \begin{equation*}
        T_{1}'
        \sqcup
        T_{2}''
        \sqcup
        T_{3}
            \quad = \quad
            \vcenter{\hbox{\scalebox{1}{\begin{tikzpicture}
                \drawlatticegrid{-1}{2}{-1}{1}{3}
                
                \fill[fill=black!25]
                    (0, 0) -- (1, -1) -- (4/3, -1) -- cycle;
                \begin{pgfonlayer}{polygon1}
                    \draw [thick] (0, 0) -- (1, -1);
                    \draw [thick] (1, -1) -- (4/3, -1);
                    \draw [thick, dashed] (0, 0) -- (4/3, -1);
                \end{pgfonlayer}

                \fill[fill=black!25]
                    (0, 0) -- (4/3, -1) -- (2, -1) -- cycle;
                \begin{pgfonlayer}{polygon1}
                    \draw [thick] (0,0) -- (4/3, -1);
                    \draw [thick] (4/3, -1) -- (2, -1);
                    \draw [thick, dashed] (0, 0) -- (2, -1);
                \end{pgfonlayer}

                \fill[fill=black!25]
                    (0, 0) -- (0, -1) -- (1, -1) -- cycle;
                \begin{pgfonlayer}{polygon1}
                    \draw [thick] (0, 0) -- (0, -1);
                    \draw [thick] (0, -1) -- (1, -1);
                \end{pgfonlayer}
            \end{tikzpicture}}}}\;.
    \end{equation*}
    Since $b \in \Z$, the matrix $\rinlinematrix{1}{-b}{0}{1}$ is
    again an automorphism of the integer lattice.  Therefore, $\ehr{T}
    = \ehr{B}$, as required.
\end{proof}

The last remaining piece in the proof of our second main result
(\cref{thm: Infinitely many triangular PRPs}) is to show that the PIPs
constructed in \cref{thm: Infinitely many PRPs} have arbitrarily large
denominators.  (The fractions in \cref{eqns: vertices of infinitely
many triangular PRPs} are often not reduced.)

To this end, we first ``regularize'' the map $(T, E) \mapsto (x, y,
z)$ in the statement of \cref{cor: lattice length of an edge in a
triangle} so that we have a well-defined invariant for equivalence
classes of rational triangles modulo automorphisms of the lattice.

\begin{definition}\label{defn: an invariant for rational triangles}
    Let $\T$ be the set of rational triangles in $\R^{2}$.  Define a
    map ${V \maps \T \to \Zp^{3}}$ as follows.  Given $T \in \T$, let
    $u, v, w$ be the respective primitive outer normals of $T$ (in
    arbitrary order).  Let $x_{0} \deftobe \det(v, w)$, $y_{0}
    \deftobe \det(w, u)$, and $z_{0} \deftobe \det(u, v)$.  Finally,
    let $(x, y, z)$ be the result of rewriting $(x_{0}, y_{0}, z_{0})$
    in increasing order, and define $V(T) \deftobe (x, y, z)$.
\end{definition}

Observe that each solution $(x,y,z) \in \SSS$ is in increasing order
by \cref{lem: Infinitely many solutions}.  Furthermore, $V(T_{xyz}) =
(x,y,z)$ by \cref{eqns: outer normals}.  The following proposition
proves that $V$ is an invariant of rational triangles under the action
of $\GL_{2}(\Z)$.

\begin{proposition}
    We have that $V(T) = V(UT)$ for all rational triangles $T \in \T$
    and lattice automorphisms $U \in \GL_{2}(\Z)$.
\end{proposition}

\begin{proof}
    Let $u, v, w$ be the respective primitive outer normals of $T$ (in
    arbitrary order).  Thus,
    \begin{equation*}
        T
            =%
            \setof{%
                a \in \R^{2} 
                \st
                \inner{u, a} \le \alpha, \,
                \inner{v, a} \le \beta, \,
                \inner{w, a} \le \gamma
            }%
    \end{equation*}
    for some $\alpha, \beta, \gamma \in \Q$.  Let $x_{0} \deftobe
    \det(v, w)$, $y_{0} \deftobe \det(w, u)$, and $z_{0} \deftobe
    \det(u, v)$.  Thus, $V(T)$ is, up to reordering, $(x_{0}, y_{0},
    z_{0})$.  On the other hand,
    \begin{align*}
        UT
            &=%
            \setof{%
                a \in \R^{2} 
                \st
                \inner{u, U^{-1}a} \le \alpha, \,
                \inner{v, U^{-1}a} \le \beta, \,
                \inner{w, U^{-1}a} \le \gamma
            } \\%
            &=%
            \setof{%
                a \in \R^{2} 
                \st
                \inner{U^{-t}u, a} \le \alpha, \,
                \inner{U^{-t}v, a} \le \beta, \,
                \inner{U^{-t}w, a} \le \gamma
            },%
    \end{align*}
    where $U^{-t}$ denotes the transpose of the inverse of $U$.  Since
    $\det(U^{-t}) = 1$, we have that $x_{0} = \det(U^{-t}v, U^{-t}w)$,
    and similar equations hold for $y_{0}$ and $z_{0}$.  Therefore,
    $V(UT)$ is also $(x_{0}, y_{0}, z_{0})$, up to reordering.
\end{proof}

We are now ready to prove our second main result (\cref{thm:
Infinitely many triangular PRPs}), which we restate here for the
convenience of the reader.

\begin{theorem}\label{thm: Infinitely many triangular PRPs - restated}
    Let $b \in \Z$ be such that $1 \le b \le 9$ and $b \ne 7$.  Then,
    for all $k \in \Z$, there exists a rational pseudointegral
    triangle with denominator $> k$ that has exactly one interior
    lattice point and $b$ boundary lattice points.
\end{theorem}

\begin{proof}
    Let $k \in \Zp$.  Since $\SSS$ is infinite and $V$ is an invariant
    of rational triangles, the set $\setof{T_{xyz} \st (x,y,z) \in
    \SSS}$ is an infinite set of pairwise-nonequivalent PIPs with
    exactly one interior lattice point and~$b$ boundary lattice
    points.  (Here, two triangles $T_{1}, T_{2} \subset \R^{2}$ are
    called equivalent if $UT_{1} = T_{2}$ for some $U \in
    \GL_{2}(\Z)$.)  Now, by a result of
    Lagarias~\&~Ziegler~\cite[Theorem~2]{LagZie1991}, there exist only
    finitely many rational polygons (up to lattice equivalence) of
    denominator $\le k$ that contain exactly one interior lattice
    point.  The result follows.
\end{proof}

\section{New Ehrhart polynomials}
\label{sec: New Ehrhart polynomials}

In our final section, we construct non-integral polygonal PIPs with a
number of boundary lattice points exceeding the number attainable by
integral polygons with the same number $i$ of interior lattice points
by a factor of $5/2$ (asymptotically in~$i$).  We repeat \cref{thm:
third main result} here for the convenience of the reader.

\begin{theorem}\label{thm: third main result - restated}
    For all $i, b \in \Zp$ such that $b \le 5i + 4$, there exists a
    rational polygonal PIP~$P$ such that $\intp{P} = i$ and $\bndp{P}
    = b$.
\end{theorem}
    
The $b = 1$ case of \cref{thm: third main result - restated} follows
from the constructions in \cref{exm: PIPs with 1 or 2 boundary
points}.  For the $b \ge 2$ cases, \cref{thm: third main result -
restated} is a corollary of the following more-precise result.

\begin{theorem}\label{thm: constructions of PIPs with fixed denominators}\mbox{}
    Let $i$ and $b$ be positive integers.
    \begin{itemize}
        \item  
        Let $d \deftobe 3$.  If $2 \le b \le 3i + 5$, then there
        exists a denominator-$d$ polygonal PIP $P$ such that $\intp{P}
        = i$ and $\bndp{P} = b$.
    
        \item  
        Let $d \deftobe 4$.  If $2 \le b \le 4i + 4$, then there
        exists a denominator-$d$ polygonal PIP $P$ such that $\intp{P}
        = i$ and $\bndp{P} = b$.
    
        \item  
        Let $d \deftobe 10$.  If $2 \le b \le 5i + 4$, then there
        exists a denominator-$d$ polygonal PIP $P$ such that $\intp{P}
        = i$ and $\bndp{P} = b$.
    \end{itemize}
\end{theorem}

\subsection*{Definition of \texorpdfstring{$\PIP{d}{i}{b}$}{the polygon}}

Before proving \cref{thm: constructions of PIPs with fixed
denominators}, we first define a convex denom\-inator-$d$ polygon
$\PIP{d}{i}{b}$ containing $i$ interior lattice points and $b$
boundary lattice points for all values of $d$, $i$, and $b$ such that
$d \in \setof{3,4,10}$ and $i$ and $b$ satisfy the corresponding
inequality in \cref{thm: constructions of PIPs with fixed
denominators}.  This theorem will then be proved once we show that the
polygon $\PIP{d}{i}{b}$ is in fact a PIP. The fact that
$\PIP{d}{i}{b}$ is a PIP appears as \cref{lem: the constructions are
PIPs} below.

We first say a few words about the limited range of the denominators
$d$ that we consider.  Roughly speaking, our polygons have a ``right
half'' that is constructed in a uniform way for all $d \in
\setof{3,4,10}$.  This half takes inspiration from Bohnert's
half-integral polygons \cite{Boh2024preprint}.  There is a clear and
natural way to generalize this right half to arbitrary denominators.
However, the ``left half'' of our polygons are \emph{ad hoc} in
certain respects.  We do not know how to generalize this half to
larger denominators.

\subsubsection*{The denominator-\texorpdfstring{$3$}{3} polygon}

\begin{sloppypar}
    Let positive $i, b \in \Z$ be such that ${2 \le b \le 3i + 5}$.  We
    define a convex denominator\nobreakdash-$3$ polygon $\PIP{3}{i}{b}$
    with $i$ interior lattice points and~$b$~boundary lattice points as
    follows:
    \begin{equation*}
        \PIP{3}{i}{b}
            \deftobe
            \conv
            \setof{%
                (i, 0),\,
                (0, \tfrac{1}{3}),\,
                (-2,-1),\,
                (b-4, -1),\,
                \bigparens{i + \tfrac{2}{3}(b-5) , -\tfrac{2}{3}}
            }.%
    \end{equation*}
    When $i = 1$, the point $\bigparens{i + \tfrac{2}{3}(b-5) ,
    -\tfrac{2}{3}}$ is not a vertex.  For all $i \ge 1$, the point $(i,
    0)$ is not a vertex in the extremal case $b = 3i + 5$.
    % When $i = 2$, no values of $(i, b)$ are attained that are not
    % attained by half-integral polygons.  However,
    When $i \ge 3$, new values of $(i,b)$ are attained by $\PIP{3}{i}{b}$
    that are not realized by integral or half-integral PIPs.  Here is the
    $i = 3$, $b = 14$ case:
    % When $i \ge
    % 2$, the polygon $\PIP{3}{i}{b}$ is a $5$-gon except in the extremal
    % cases $b = 2$ and $b = 3i + 5$ when it is a $4$-gon.  When $b = 3i+5$,
    % the point $(i, 0)$ is not extremal and may be omitted from the
    % right-hand side above.
\end{sloppypar}

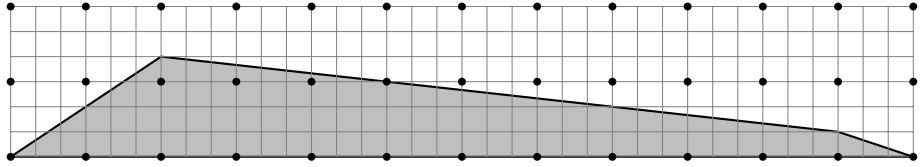
\begin{center}
    \newcommand{\intpvalue}{3}
    \newcommand{\bndpvalue}{14}
    \begin{tikzpicture}
        \drawlatticegrid{-2}{10}{-1}{1}{3}
        \drawfivegon{(\intpvalue,0)}{(0,1/3)}{(-2,-1)}{(\bndpvalue-4,-1)}{(\intpvalue+\bndpvalue*2/3-5*2/3,-2/3)}
    \end{tikzpicture}
\end{center}

\subsubsection*{The denominator-\texorpdfstring{$4$}{4} polygon} 

\begin{sloppypar}
    Let positive $i, b \in \Z$ be such that ${2 \le b \le 4i + 4}$.  We
    define a convex denominator\nobreakdash-$4$ polygon $\PIP{4}{i}{b}$
    with $i$ interior lattice points and~$b$~boundary lattice points as
    follows:
    \begin{equation*}
        \PIP{4}{i}{b}
            \deftobe
            \conv
            \setof{%
                (i, 0),\,
                (0, \tfrac{1}{4}),\,
                (-1, -\tfrac{1}{2}),\,
                (-1,-1),\,
                (b-3, -1),\,
                \bigparens{i + \tfrac{3}{4}(b-4) , -\tfrac{3}{4}}
            }.%
        \end{equation*}
    % However, when $i \ge 2$, new values of
    % $(i,b)$ are attained that were not reached by integral or
    % half-integral PIPs.)  
    When $i = 1$, the point $\bigparens{i + \tfrac{3}{4}(b-4) ,
    -\tfrac{3}{4}}$ is not a vertex.  For all $i \ge 1$, the point $(i,
    0)$ is not a vertex in the extremal case $b = 4i + 4$.  When $i \ge
    2$, new values of $(i,b)$ are attained by $\PIP{4}{i}{b}$ that are not
    realized by known PIPs of denominator $\le 3$.  Here is the $i = 2$,
    $b = 12$ case (which is extremal):
    % The polygon
    % $\PIP{4}{i}{b}$ is a $6$-gon except in the extremal cases $b = 2$ and
    % $b = 4i + 4$ when it is a $5$-gon.  When $b = 4i+4$, the point $(i,
    % 0)$ is not extremal and may be omitted from the right-hand side above.
\end{sloppypar}

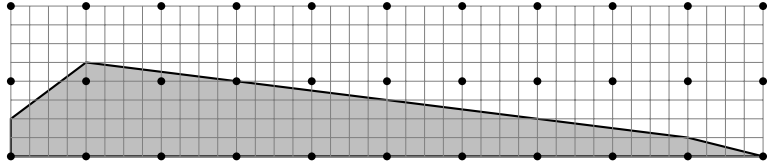
\begin{center}
    \newcommand{\intpvalue}{2}
    \newcommand{\bndpvalue}{12}
    \begin{tikzpicture}
        \drawlatticegrid{-1}{9}{-1}{1}{4}
        \drawsixgon{(\intpvalue,0)}{(0,1/4)}{(-1,-1/2)}{(-1,-1)}{(\bndpvalue-3,-1)}{(\intpvalue+3/4*\bndpvalue-3/4*4,-3/4)}
    \end{tikzpicture}
\end{center}

\subsubsection*{The denominator-\texorpdfstring{$10$}{10} polygon}
Let positive $i, b \in \Z$ be such that ${2 \le b \le 5i + 4}$.  We
define a convex denominator\nobreakdash-$10$ polygon $\PIP{10}{i}{b}$
with $i$ interior lattice points and~$b$~boundary lattice points as
follows: 
\begin{equation*}
    \PIP{10}{i}{b}
        \deftobe
        \conv
        \setof{%
            (i, 0),\,
            (0, \tfrac{1}{5}),\,
            (-\tfrac{3}{2}, -1),\,
            (b-3, -1),\,
            \bigparens{i + \tfrac{4}{5}(b-4), -\tfrac{4}{5}}
        }.%
\end{equation*}
When $i = 1$, the point $\bigparens{i + \tfrac{4}{5}(b-4),
-\tfrac{4}{5}}$ is not a vertex.  For all $i \ge 1$, the point $(i,
0)$ is not a vertex in the extremal case $b = 5i + 4$.  Here is the $i
= 2$, $b = 14$ case (which is extremal):
% The polygon $\PIP{10}{i}{b}$ is a $5$-gon except in the extremal cases
% $b = 2$ and $b = 5i + 4$ when it is a $4$-gon.  When $b = 5i+4$, the
% point $(i, 0)$ is not extremal and may be omitted from the right-hand
% side above.

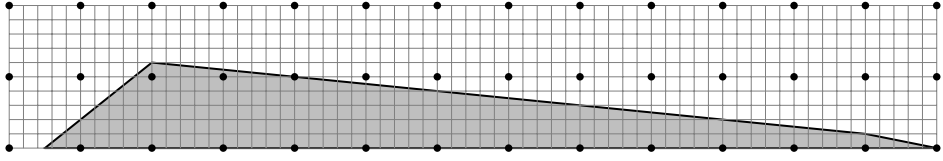
\begin{center}
    \newcommand{\intpvalue}{2}
    \newcommand{\bndpvalue}{14}
    \resizebox{\textwidth}{!}{
        \begin{tikzpicture}
            \drawlatticegrid{-2}{11}{-1}{1}{5}
            \drawfivegon{(\intpvalue,0)}{(0,1/5)}{(-3/2,-1)}{(\bndpvalue-3,-1)}{(\intpvalue+4/5*\bndpvalue-4/5*4,-4/5)}
        \end{tikzpicture}
    }
\end{center}

\subsection*{Proof that \texorpdfstring{$\PIP{d}{i}{b}$}{the defined polygon} is a PIP} 
Let $d \in \setof{3,4,10}$.  It is straightforward to confirm that the
polygon $\PIP{d}{i}{b}$ defined above has denominator $d$ and contains
$i$ interior lattice points and $b$ boundary lattice points, provided
that $i$ and $b$ satisfy the corresponding inequalities.  In order to
prove \cref{thm: constructions of PIPs with fixed denominators} (and
hence \cref{thm: third main result - restated}), it remains only to
show that $\PIP{d}{i}{b}$ is in fact a PIP. We complete this last step
in the proof of \cref{thm: constructions of PIPs with fixed
denominators} with the following lemma.

\begin{lemma}\label{lem: the constructions are PIPs} %
    Let $d \in \setof{3,4,10}$, and let $(i, b) \in \Zp^{2}$ be such 
    that
    \begin{equation*}
        \begin{dcases*}
            2 \le b \le 3i+5 
                & if $d = 3$, \\
            2 \le b \le 4i+4
                & if $d = 4$, \\
            2 \le b \le 5i+4
                & if $d = 10$.
        \end{dcases*}
    \end{equation*}
    Then the polygon $\PIP{d}{i}{b}$ defined above is a polygonal PIP.
\end{lemma}

\begin{proof}
    For the $d = 3$ case, let
    \begin{align*}
        A
            & \deftobe
            \conv\setof{(i,0), (i, -\tfrac{2}{3}), (1, -1)}, \\
        B
            & \deftobe
            \conv\setof{(1, -1), (i, -\tfrac{2}{3}), (i, -1)}
            \setminus
            \conv\setof{(1, -1), (i, -\tfrac{2}{3})}, \\
        C 
            & \deftobe
            \conv\setof{(0,0), (1,0), (0, \tfrac{1}{3})}, \\
        D
            & \deftobe
            \conv\setof{(0,0), (0, \tfrac{1}{3}), (-2, 1)}
            \setminus 
            \conv\setof{(0,0), (0, \tfrac{1}{3})}.
    \end{align*}
    (Observe that, in the $i = 1$ case, $A$ is the line segment
    $\conv\setof{(1,0), (1, -1)}$, and $B = \emptyset$.) As a running 
    example, here is the $i = 3$ case:
    \begin{center}
        \newcommand{\intpvalue}{3}
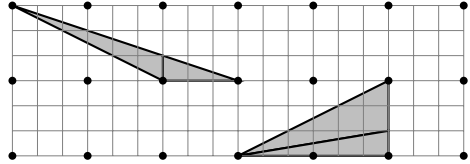
        \begin{tikzpicture}
            \drawlatticegrid{-2}{4}{-1}{1}{3}
            \drawtriangle{(\intpvalue,0)}{(\intpvalue,-2/3)}{(1,-1)}
            \drawtriangle{(1,-1)}{(\intpvalue,-2/3)}{(\intpvalue,-1)}
            \drawtriangle{(0,0)}{(1,0)}{(0,1/3)}
            \drawtriangle{(0,0)}{(0,1/3)}{(-2, 1)}
        \end{tikzpicture}
    \end{center}
    Thus, $(A \cup B) \sqcup (C \cup D)$ is a disjoint union of
    \emph{lattice} polygons (or of a lattice polygon and a lattice
    line segment, in the $i = 1$ case), so the Ehrhart function of
    this union is a polynomial.  We will now apply affine lattice
    automorphisms to the pairwise disjoint partially-open polygons $A,
    B, C, D$ in such a way that the images of these partially-open
    polygons under their respective affine lattice automorphisms are
    still pairwise disjoint.  It will then follow that the union of
    these images has the same polynomial Ehrhart function as $(A \cup
    B) \sqcup (C \cup D)$ has.
    
    To this end, let $D' \deftobe \rinlinematrix{1}{0}{1}{1} D$, so
    that $D'$ is the result of shearing $D$ parallel to the $y$-axis
    by a lattice automorphism, and let $B' \deftobe B + (0,1)$, so
    that $B'$ is a lattice translate of~$B$.
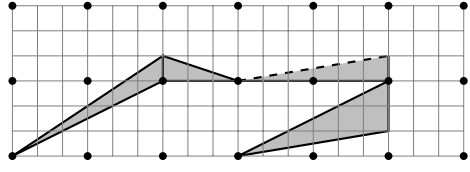
    \begin{center}
        \newcommand{\intpvalue}{3}
        \begin{tikzpicture}
            \drawlatticegrid{-2}{4}{-1}{1}{3}
            \drawtriangle{(\intpvalue,0)}{(\intpvalue,-2/3)}{(1,-1)}
            \fill[fill=black!25]
                (1,0) -- (\intpvalue,1/3) -- (\intpvalue,0) -- cycle;
            \draw [thick, dashed] (1,0) -- (\intpvalue,1/3);
            \draw [thick] (\intpvalue,1/3) -- (\intpvalue,0);
            \draw [thick] (\intpvalue,0) -- (1,0);
            \drawtriangle{(0,0)}{(1,0)}{(0,1/3)}
            \drawtriangle{(0,0)}{(0,1/3)}{(-2, -1)}
        \end{tikzpicture}
    \end{center}
    Finally, let $B''$ and $A'$ be the respective results of shearing
    $B'$ and $A$ parallel to the $x$-axis as follows: $B'' \deftobe
    \cinlinematrix{1}{-3i}{0}{1} B'$ and $A' \deftobe
    \cinlinematrix{1}{5-b}{0}{1} A$, where, by hypothesis, ${-3i \le 5
    - b \le 3}$.  (For the running example, we now show the $i = 3$,
    $b = 6$ case.)
    \begin{center}
        \newcommand{\intpvalue}{3}
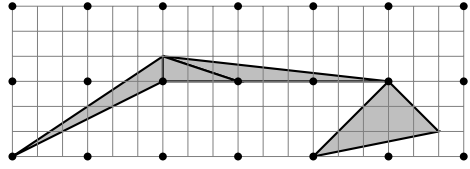
        \begin{tikzpicture}
            \drawlatticegrid{-2}{4}{-1}{1}{3}
            \drawtriangle{(\intpvalue,0)}{(\intpvalue+2/3, -2/3)}{(2,-1)}
            \drawtriangle{(1,0)}{(0,1/3)}{(\intpvalue,0)}
            \drawtriangle{(0,0)}{(1,0)}{(0,1/3)}
            \drawtriangle{(0,0)}{(0,1/3)}{(-2, -1)}
        \end{tikzpicture}
    \end{center}
    (Note that the inequality $5-b \le 3$ ensures that the image of
    $(1,-1)$ in $A'$ does not pass to the left of the point $(-2, -1)$
    in $D'$, while the inequality $-3i \le 5- b$ ensures that the
    image of the segment $\conv\setof{(i,-\tfrac{2}{3}), (i,
    \tfrac{1}{3})}$ is convex about $0$ after the action of the
    piecewise linear map $x \mapsto \cinlinematrix{1}{-3i}{0}{1} x$
    for $x \in B'$, $x \mapsto \cinlinematrix{1}{5-b}{0}{1} x$ for $x
    \in A$.)  Thus, the Ehrhart function of $A' \cup B'' \cup C \cup
    D'$ is a polynomial.  Now, $\PIP{3}{i}{b}$ is the union of $A'
    \cup B'' \cup C \cup D'$ and the lattice polygon
    $\conv\setof{(0,0), (i, 0), (b-4, -1), (-2, -1)}$.
    \begin{center}
        \newcommand{\intpvalue}{3}
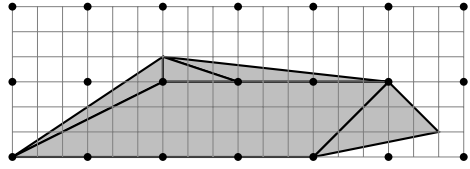
        \begin{tikzpicture}
            \drawlatticegrid{-2}{4}{-1}{1}{3}
            \drawtriangle{(\intpvalue,0)}{(\intpvalue+2/3, -2/3)}{(2,-1)}
            \drawtriangle{(1,0)}{(0,1/3)}{(\intpvalue,0)}
            \drawtriangle{(0,0)}{(1,0)}{(0,1/3)}
            \drawtriangle{(0,0)}{(0,1/3)}{(-2, -1)}
            \drawfourgon{(0,0)}{(3,0)}{(2,-1)}{(-2,-1)}
        \end{tikzpicture}
    \end{center}
    Moreover, the intersection of these two regions is a lattice
    polygonal path.  Therefore, the Ehrhart function of
    $\PIP{3}{i}{b}$ is a polynomial, as claimed.
        
    We proceed similarly for the $d = 4$ and $d = 10$ cases.  Our
    running examples will again show the corresponding $i = 3$, $b =
    6$ cases.
    
    For the $d = 4$ case, let
    \begin{align*}
        A
            & \deftobe
            \conv\setof{(i,0), (i, -\tfrac{3}{4}), (1, -1)}, \\
        B
            & \deftobe
            \conv\setof{(1, -1), (i, -\tfrac{3}{4}), (i, -1)}
            \setminus
            \conv\setof{(1, -1), (i, -\tfrac{3}{4})}, \\
        C 
            & \deftobe
            \conv\setof{(0,0), (1,0), (0, \tfrac{1}{4})}, \\
        D
            & \deftobe
            \conv\setof{(0,0), (0, \tfrac{1}{4}), (-1, \tfrac{1}{2})}
            \setminus 
            \conv\setof{(0,0), (0, \tfrac{1}{4})}, \\
        E
            & \deftobe
            \conv\setof{(0,0), (-1, \tfrac{1}{2}), (-3, 1)}
            \setminus 
            \conv\setof{(0,0), (-1, \tfrac{1}{2})}.
    \end{align*}
    \begin{center}
        \newcommand{\intpvalue}{3}
        \begin{tikzpicture}
            \drawlatticegrid{-3}{5}{-1}{1}{4}
            \drawtriangle{(\intpvalue,0)}{(\intpvalue,-3/4)}{(1,-1)}
            \drawtriangle{(1,-1)}{(\intpvalue,-3/4)}{(\intpvalue,-1)}
            \drawtriangle{(0,0)}{(1,0)}{(0,1/4)}
            \drawtriangle{(0,0)}{(0,1/4)}{(-1,1/2)}
            \drawtriangle{(0,0)}{(-1,1/2)}{(-3,1)}
        \end{tikzpicture}
    \end{center}
    Let $E' \deftobe \rinlinematrix{-1}{-4}{1}{3}E$ and let $B'
    \deftobe B + (0,1)$.
    \begin{center}
        \newcommand{\intpvalue}{3}
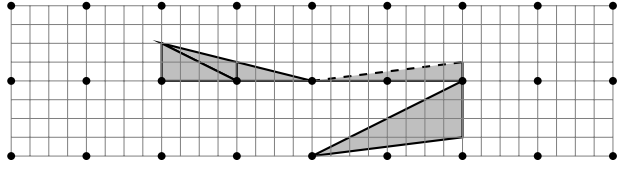
        \begin{tikzpicture}
            \drawlatticegrid{-3}{5}{-1}{1}{4}
            \drawtriangle{(\intpvalue,0)}{(\intpvalue,-3/4)}{(1,-1)}
            \fill[fill=black!25]
                (1,0) -- (\intpvalue,1/4) -- (\intpvalue,0) -- cycle;
            \draw [thick, dashed] (1,0) -- (\intpvalue,1/4);
            \draw [thick] (\intpvalue,1/4) -- (\intpvalue,0);
            \draw [thick] (\intpvalue,0) -- (1,0);
            \drawtriangle{(0,0)}{(1,0)}{(0,1/4)}
            \drawtriangle{(0,0)}{(0,1/4)}{(-1,1/2)}
            \drawtriangle{(0,0)}{(-1,1/2)}{(-1,0)}
        \end{tikzpicture}
    \end{center}
    Let $F \deftobe \rinlinematrix{1}{0}{1}{1}\parens{D \cup E'}$.
    \begin{center}
        \newcommand{\intpvalue}{3}
        \begin{tikzpicture}
            \drawlatticegrid{-3}{5}{-1}{1}{4}
            \drawtriangle{(\intpvalue,0)}{(\intpvalue,-3/4)}{(1,-1)}
            \fill[fill=black!25]
                (1,0) -- (\intpvalue,1/4) -- (\intpvalue,0) -- cycle;
            \draw [thick, dashed] (1,0) -- (\intpvalue,1/4);
            \draw [thick] (\intpvalue,1/4) -- (\intpvalue,0);
            \draw [thick] (\intpvalue,0) -- (1,0);
            \drawtriangle{(0,0)}{(1,0)}{(0,1/4)}
            \drawtriangle{(0,0)}{(0,1/4)}{(-1,-1/2)}
            \drawtriangle{(0,0)}{(-1,-1/2)}{(-1,-1)}
        \end{tikzpicture}
    \end{center}
    Finally, let $B'' \deftobe \cinlinematrix{1}{-4i}{0}{1} B'$ and
    $A' \deftobe \cinlinematrix{1}{4-b}{0}{1} A$, where, by 
    hypothesis, $-4i \le 4 - b \le 2$.
    \begin{center}
        \newcommand{\intpvalue}{3}
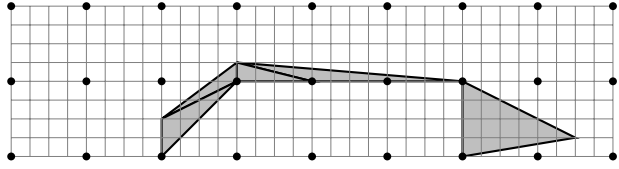
        \begin{tikzpicture}
            \drawlatticegrid{-3}{5}{-1}{1}{4}
            \drawtriangle{(\intpvalue,0)}{(\intpvalue+3/2,-3/4)}{(3,-1)}
            \drawtriangle{(1,0)}{(0,1/4)}{(\intpvalue,0)}
            \drawtriangle{(0,0)}{(1,0)}{(0,1/4)}
            \drawtriangle{(0,0)}{(0,1/4)}{(-1,-1/2)}
            \drawtriangle{(0,0)}{(-1,-1/2)}{(-1,-1)}
        \end{tikzpicture}
    \end{center}
    Then $\PIP{4}{i}{b} = \conv\parens{A' \cup B'' \cup C \cup F}$
    and, as in the previous case, the Ehrhart function of
    $\PIP{4}{i}{b}$ is a polynomial.

    For the $d = 10$ case, let
    \begin{align*}
        A
            & \deftobe
            \conv\setof{(i,0), (i, -\tfrac{4}{5}), (1, -1)}, \\
        B
            & \deftobe
            \conv\setof{(1, -1), (i, -\tfrac{4}{5}), (i, -1)}
            \setminus
            \conv\setof{(1, -1), (i, -\tfrac{4}{5})}, \\
        C 
            & \deftobe
            \conv\setof{(0,0), (1,0), (0, \tfrac{1}{5})}, \\
        D
            & \deftobe
            \conv\setof{(0,0), (0, \tfrac{1}{5}), (-\tfrac{3}{2}, \tfrac{1}{2})}
            \setminus 
            \conv\setof{(0,0), (0, \tfrac{1}{5})}, \\
        E
            & \deftobe
            \conv\setof{(0,0), (-\tfrac{3}{2}, \tfrac{1}{2}), (-4, 1)}
            \setminus 
            \conv\setof{(0,0), (-\tfrac{3}{2}, \tfrac{1}{2})}.
    \end{align*}
    \begin{center}
        \newcommand{\intpvalue}{3}
        \begin{tikzpicture}
            \drawlatticegrid{-4}{5}{-1}{1}{5}
            \drawtriangle{(\intpvalue,0)}{(\intpvalue,-4/5)}{(1,-1)}
            \drawtriangle{(1,-1)}{(\intpvalue,-4/5)}{(\intpvalue,-1)}
            \drawtriangle{(0,0)}{(1,0)}{(0,1/5)}
            \drawtriangle{(0,0)}{(0,1/5)}{(-3/2,1/2)}
            \drawtriangle{(0,0)}{(-3/2,1/2)}{(-4,1)}
        \end{tikzpicture}
    \end{center}
    Let $E' \deftobe \rinlinematrix{-2}{-9}{1}{4}E$ and let $B'
    \deftobe B + (0,1)$.
    \begin{center}
        \newcommand{\intpvalue}{3}
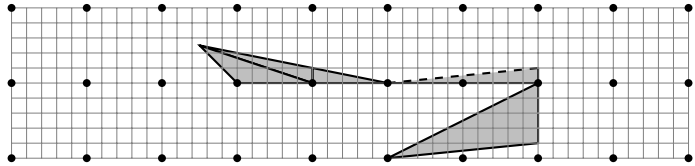
        \begin{tikzpicture}
            \drawlatticegrid{-4}{5}{-1}{1}{5}
            \drawtriangle{(\intpvalue,0)}{(\intpvalue,-4/5)}{(1,-1)}
            \fill[fill=black!25]
                (1,0) -- (\intpvalue,1/5) -- (\intpvalue,0) -- cycle;
            \draw [thick, dashed] (1,0) -- (\intpvalue,1/5);
            \draw [thick] (\intpvalue,1/5) -- (\intpvalue,0);
            \draw [thick] (\intpvalue,0) -- (1,0);
            \drawtriangle{(0,0)}{(1,0)}{(0,1/5)}
            \drawtriangle{(0,0)}{(0,1/5)}{(-3/2,1/2)}
            \drawtriangle{(0,0)}{(-3/2,1/2)}{(-1,0)}
        \end{tikzpicture}
    \end{center}
    Let $F \deftobe \rinlinematrix{1}{0}{1}{1}\parens{D \cup E'}$.
    \begin{center}
        \newcommand{\intpvalue}{3}
        \begin{tikzpicture}
            \drawlatticegrid{-4}{5}{-1}{1}{5}
            \drawtriangle{(\intpvalue,0)}{(\intpvalue,-4/5)}{(1,-1)}
            \fill[fill=black!25]
                (1,0) -- (\intpvalue,1/5) -- (\intpvalue,0) -- cycle;
            \draw [thick, dashed] (1,0) -- (\intpvalue,1/5);
            \draw [thick] (\intpvalue,1/5) -- (\intpvalue,0);
            \draw [thick] (\intpvalue,0) -- (1,0);
            \drawtriangle{(0,0)}{(1,0)}{(0,1/5)}
            \drawtriangle{(0,0)}{(0,1/5)}{(-3/2,-1)}
            \drawtriangle{(0,0)}{(-3/2,-1)}{(-1,-1)}
        \end{tikzpicture}
    \end{center}
    Finally, let $B'' \deftobe \cinlinematrix{1}{-5i}{0}{1} B'$ and
    $A' \deftobe \cinlinematrix{1}{4-b}{0}{1} A$, where, by
    hypothesis, $-5i \le 4-b \le 2$.
    \begin{center}
        \newcommand{\intpvalue}{3}
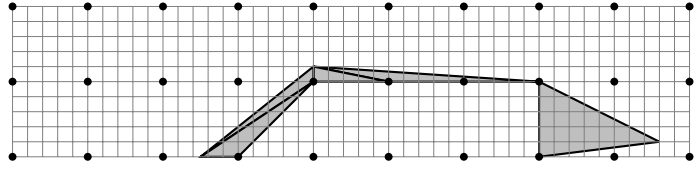
        \begin{tikzpicture}
            \drawlatticegrid{-4}{5}{-1}{1}{5}
            \drawtriangle{(\intpvalue,0)}{(\intpvalue+8/5,-4/5)}{(3,-1)}
            \drawtriangle{(1,0)}{(0,1/5)}{(\intpvalue,0)}
            \drawtriangle{(0,0)}{(1,0)}{(0,1/5)}
            \drawtriangle{(0,0)}{(0,1/5)}{(-3/2,-1)}
            \drawtriangle{(0,0)}{(-3/2,-1)}{(-1,-1)}
        \end{tikzpicture}
    \end{center}
    As in the previous cases, the Ehrhart function of $\PIP{5}{i}{b} =
    \conv\parens{A' \cup B'' \cup C \cup F}$ is a polynomial.
\end{proof}

\section*{Acknowledgements}

We thank Martin Bohnert for helpful conversations, and Will Jagy for
pointing us to the reference \cite{Hur1963}.

The work of the second author was supported by a grant from the Simons
Foundation (711898, J.S.W.).  Exploratory calculations were done using
the Magma \cite{magma1997} and SageMath \cite{sagemath} computer
algebra systems.

\appendix

\section{A generalization of \texorpdfstring{\cref{lem: number-theory
lemma for triangular PRPs}}{the number-theory lemma}}

\Cref{lem: number-theory lemma for triangular PRPs} is the $n = 3$
case of a generalization that we prove here for all $n \ge 2$.
Unfortunately, for $n \ge 4$, this result does not have the direct
relevance to $n$-gonal PRPs that it had in the $n = 3$ case.  In
particular, this result cannot be used to show that $\bndp{P} \le 9$
for every pseudoreflexive $n$-gon, so far as we know.  Nonetheless,
the result may be of independent interest.

\begin{theorem}\label{thm: number-theory lemma for general n}
    Let $n \ge 2$ be an integer, and let $x_{1}, \dotsc, x_{n} \in
    \Zp$ be such that the ratio
    \begin{equation}\label{eq:1}
        b
        \; \deftobe \;
        \parens{\sum_{i = 1}^{n} x_{i}}^{2}
        \Bigg\slash 
        \prod_{i=1}^{n} x_{i}
    \end{equation}
    is an integer.  Then $b \le n^{2}$.
\end{theorem}

\begin{proof} 
    Without loss of generality, $x_{1} \le \dotsb \le x_{n}$.  As in
    the proof of \cref{lem: number-theory lemma for triangular PRPs},
    we first use Vieta jumping to show that we may further assume
    without loss of generality that $x_{n} \le \sum_{i=1}^{n-1}
    x_{i}$.  For, \cref{eq:1} is equivalent to the condition that
    $x_{n}$ is a root of the polynomial $f(t) \in \Z[t]$ defined by
    \begin{equation*}
       f(t)
           \deftobe
           t^{2} 
           - 
           \parens{
               b \prod_{i=1}^{n-1} x_{i} 
               -
               2 \sum_{i=1}^{n-1} x_{i}
           } t 
           + 
           \parens{
               \sum_{i=1}^{n-1} x_{i}
           }^{2}.
    \end{equation*}
    Letting $x_{n}'$ be the other root of $f(t)$, we get from the
    integrality of the linear coefficient that $x_{n}' \in \Z$.
    Furthermore, from the positivity of the constant coefficient of
    $f(t)$, we get that $x_{n}x_{n}' =
    \smash{\bigparens{\sum_{i=1}^{n-1} x_{i}}^{2}} \ge 1$ and hence
    $x_{n}' \ge 1$.  (This is where we use the hypothesis that $n \ge
    2$.)  If $x_{n} \le x_{n}'$, then we already have that $x_{n} \le
    \sum_{i=1}^{n-1} x_{i}$.  Otherwise, we have a solution $(x_{1},
    \dotsc, x_{n-1}, x_{n}')$ to \cref{eq:1} with $1 \le x_{n}' <
    \sum_{i=1}^{n-1} x_{i} < x_{n}$.  In that case, reordering $x_{1},
    \dotsc, x_{n-1}, x_{n}'$ in increasing order, and then reapplying
    this reduction step a finite number of times, we eventually arrive
    at a solution of the desired form.

    Now, given a solution $(x_{1}, \dotsc, x_{n}) \in
    \smash{\Zp^{n}}$ to \cref{eq:1} with $x_{1} \le \dotsb \le x_{n}
    \le \sum_{i=1}^{n-1} x_{i}$, we find that
    \begin{align*}
       b
           &=%
           \frac{
               \bigparens{
                   \sum_{i=1}^{n} 
                   x_{i}
               }^{2}
           }{
               \prod_{i=1}^{n} x_{i}
           } 
           \le%
           2
           \frac{
               \bigparens{
                   \sum_{i=1}^{n-1} 
                   x_{i}
               }
               \bigparens{
                   \sum_{i=1}^{n}
                   x_{i}
               }
           }{
               \prod_{i=1}^{n} x_{i}
           } 
           \le %
           2
           \Biggparens{
               \sum_{i=1}^{n-1}
               \prod_{\substack{1 \le j \le n-1, \\ j \ne i}}
               \frac{
                   1
               }{
                   x_{j}
               }
           }
           \parens{
               \sum_{i=1}^{n} 
               \frac{x_{i}}{x_{n}}
           }.
    \end{align*}
    Thus, if $x_{n-1} \ge 2$, then $b \le 2 \bigparens{(n-2)
    \frac{1}{2} + 1}n = n^{2}$.  We may therefore suppose that $x_{1}
    = \dotsb = x_{n-1} = 1$.  If also $x_{n} = 1$, then the claim to be
    proved is immediate, so we suppose that $2 \le x_{n} \le \sum_{i =
    1}^{n-1} x_{i} = n-1$.  In this last case,
    \begin{equation*}
        b 
            =
            \frac{(n-1 + x_{n})^{2}}{x_{n}}
            = 
            (n-1 + x_{n}) \parens{\frac{n-1}{x_{n}} + 1} 
            \le
            2 (n-1) \parens{\frac{n-1}{2} + 1} 
            =
            n^{2} - 1. \qedhere
    \end{equation*}
\end{proof}

% \bibliography{references}
% \bibliographystyle{amsalpha-fi-arxlast}

\providecommand{\bysame}{\leavevmode\hbox to3em{\hrulefill}\thinspace}
\providecommand{\MR}{\relax\ifhmode\unskip\space\fi MR }
\providecommand{\MRhref}[2]{%
  \href{http://www.ams.org/mathscinet-getitem?mr=#1}{#2}
}
\providecommand{\href}[2]{#2}

\end{document}